\documentclass[a4paper,11pt]{article}
\usepackage[pagewise]{lineno}
\usepackage{amsmath}
\usepackage{amssymb}
\usepackage{mathrsfs}
\usepackage{amsfonts}
\usepackage{amsthm}
\usepackage{pstricks, pst-node, pst-text, pst-3d,psfrag,ulem}
\usepackage{graphicx}
\usepackage{indentfirst}
\usepackage{enumerate}
\usepackage{subfigure}
\usepackage{cite}
\usepackage[colorlinks=true]{hyperref}
\hypersetup{urlcolor=blue, citecolor=red}
\usepackage{soul}
\usepackage{footmisc}
\usepackage{authblk}

\theoremstyle{plain}
\newtheorem{thm}{Theorem}[section]
\newtheorem{prop}[thm]{Proposition}
\newtheorem{cor}[thm]{Corollary}
\newtheorem{lem}[thm]{Lemma}
\newtheorem{thmm}{Theorem}[subsection]
\theoremstyle{definition}
\newtheorem{defn}[thm]{Definition}

\theoremstyle{remark}
\newtheorem{rmk}[thm]{\textbf{Remark}}

\newtheorem*{hypo*}{\textbf{Hypothesis}}

\numberwithin{equation}{section}

\usepackage{authblk}

\makeatletter 
\@addtoreset{equation}{section}
\makeatother  
\renewcommand\theequation{\oldstylenums{\thesection}%
	.\oldstylenums{\arabic{equation}}}

\begin{document}
	
	\title{Statistical behavior of systems with nested invariant cones}

\author[1]{Xu Cheng\footnote{supported by the China Scholarship Council (No.202406840118).}}
\author[2]{Yufeng Zhang\footnote{Corresponding author: yufeng27@ualberta.ca (Y. Zhang). Supported by the National Natural Science Foundation of China (No.12401240).}} 
\author[3]{Dun Zhou\footnote{supported by the National Natural Science Foundation of China (Nos.12671213,12331006).}}
	
    \affil[1,3]{\footnotesize{School of Mathematics and Statistics, Nanjing University of Science and Technology, Nanjing, Jiangsu 210094, P. R. China}}
	\affil[2]{\footnotesize{Department of Mathematical and Statistical Sciences, University of Alberta, Edmonton, AB T6G 2G1, Canada}}

\date{}
	\renewcommand{\thefootnote}{\fnsymbol{footnote}}
\maketitle

\setcounter{footnote}{0}
\renewcommand{\thefootnote}{\arabic{footnote}}
	\maketitle
	
\begin{abstract}
Nested invariant cones (NICs) encode hierarchical order and
oscillation structures in many finite and infinite dimensional
dynamical systems. We investigate how this hierarchy constrains the
Birkhoff center and the supports of invariant probability measures.
For eventually compact, dissipative semiflows that are uniformly
eventually strongly monotone with respect to NICs, we prove that every
connected component $B$ of the Birkhoff center lies in a single cone
layer: there exists $j>0$ such that $B$ is unordered with respect to
all lower-level cones and strongly ordered with respect to all cones
at level $j$ or higher. The same conclusion holds for every connected
component of the support of an invariant probability measure.

Under an additional transversality condition involving a
codimension-$d$ linear subspace, each such component admits a
homeomorphic embedding into $\mathbb R^d$. If $d=1$ or $2$ and the
restriction of the semiflow to its global attractor extends to a
flow, then the system has zero topological entropy, independently of
the dimension of the original phase space.

We apply the theory to bidirectional cyclic feedback systems and
scalar parabolic equations on the circle. In the parabolic case, the
natural infinite family of zero-number NICs fails to be uniformly
eventually strongly monotone, even for the heat equation; we overcome
this obstruction by constructing a finite family of perturbed NICs.
In both applications, every connected component of the Birkhoff
center admits a planar embedding, and the topological entropy is zero.
	\end{abstract}

	\section{Introduction}

    Invariant cones provide a fundamental geometric framework for the study of monotone dynamical systems. The invariance of a cone under the dynamics naturally ensures that the order relation induced by the cone is preserved along trajectories. This monotonicity property lies at the heart of the structural theory of monotone systems. In the classical monotone systems, Hirsch \cite{H82-1,H85-2,H88-S,HS05} investigated the generic asymptotic behavior of such systems through the partial order induced by a closed convex cone. This framework was subsequently extended to systems preserving a cone of finite rank, and generic asymptotic behavior has been established \cite{FWW19,FO91,Z-22}. Those classical theories are primarily formulated in terms of a single cone.

    However, in many high and infinite dimensional systems, the underlying order structure is more naturally encoded by a nested hierarchy of invariant cones. Representative examples include tridiagonal competitive-cooperative systems, cyclic feedback systems, and scalar parabolic equations either on the circle or on a bounded interval with separated boundary conditions \cite{An,MalletParetSmith1990,Smillie1984,MalletParetSell1996a,MalletParetSell1996b,Matano,Matano2}.
    Following Tere\v{s}\v{c}\'{a}k \cite{Tl}, the nested invariant cones (NICs for short) are described as $$\{0\}=C^0\subset C^1\subset\cdots\subset C^N=X,$$ where $C^i$ is a invariant cone with rank $k_i$ and $k_i$ is increasing (see Definition \ref{NIC}).
    The nested hierarchy distinguishes different oscillation levels of the dynamics and is reflected in the linearized flow through a decomposition into corresponding invariant subspaces. The cone ranks are closely related to the cumulative dimensions of these subspaces \cite{FO91,Tl}. This structure also constrains invariant manifolds and connecting orbits, particularly through restrictions on Morse indices and intersections between stable and unstable manifolds \cite{Fusco1987,FOT,Joly,Tl,Margaliot-21}. Thus, NICs should not be viewed merely as an increasing family of cones; its layers jointly encode dynamical information that cannot be captured by any single cone. This motivates the study of dynamical systems with NICs.

 A fundamental objective in the theory of dynamical systems is to understand the global structure of the dynamics, including its statistical properties, recurrent behavior, dynamical complexity, and structural stability. For several important classes of systems admitting NICs structures, including scalar parabolic equations on the circle \cite{Joly,Oliva} and bidirectional cyclic feedback systems \cite{XCYWDZ,MPer}, generic Morse-Smale properties and related structural stability results have been established. However, beyond these equation-specific settings, fundamental questions about the global dynamics of general systems with NICs, in particular, their Morse-Smale properties and statistical behavior remain largely open. The Birkhoff center, studied in this paper, defined as the closure of the set of recurrent points, provides a natural starting point for studying the global structure of a system. Moreover, the support of every invariant probability measure is contained in the Birkhoff center \cite{J20,M12}, making it a natural bridge between topological recurrence and statistical dynamics.

Motivated by this connection, we investigate how the NICs
$\{C^i\}_{i=0}^{N}$ constrain the order and topological structure of
the Birkhoff center, thereby providing a structural description of
the recurrent and statistical dynamics of the system. Following Tere\v{s}\v{c}\'{a}k \cite{Tl}, the semiflow $\Phi_t$ generated by systems with NICs $\{C^i\}_{i=0}^{N}$, preserves a series of orders ``$\simeq_{C^i}$'' induced by $C^i$, that is, $\Phi_t(x)\simeq_{C^i}\Phi_t(y)$ whenever $x\simeq_{C^i} y$ and $t\ge t^i_*$ (see Definition \ref{NIC}). Throughout the paper, we write $x\simeq_{C^i} y$ if $y-x\in{C^i}$. A subset $Y$ is called unordered with $C^i$ if none of its points are related by ``$\simeq_{C^i}$''; ordered with $C^i$ if any of its points are related by ``$\simeq_{C^i}$''.

One of our main results provides an order-theoretic description of
the Birkhoff center $\mathcal B(\Phi)$. Under the assumptions of
compactness, dissipativity, and uniform eventual strong monotonicity,
we prove that every connected component $B$ of $\mathcal B(\Phi)$ is
confined to a single layer of the nested cone hierarchy. More
precisely, there exists an index $j>0$ such that $B$ is unordered
with respect to $\{C^i\}_{i=0}^{j-1}$ and strongly ordered with
respect to $\{C^i\}_{i=j}^{N}$; see
Theorem~\ref{Thm_2.3}. Equivalently, for any two distinct points
$x,y\in B$,
\[
y-x\in\operatorname{Int}C^j\setminus C^{j-1}.
\]
Since the Birkhoff center contains all recurrent dynamics, including
equilibria, periodic orbits, and nonperiodic recurrent sets, this
result reveals a global order structure that cannot be obtained by
studying individual orbits alone. In particular, the recurrent
dynamics within each connected component cannot spread across
different layers of the nested cone hierarchy.

A notable feature shared by many important systems admitting NICs is the presence of an additional hyperplane structure associated with a finite-codimensional linear subspace. Representative examples include tridiagonal competitive--cooperative systems, cyclic feedback systems, and scalar parabolic equations either on the circle or on a bounded interval with separated boundary conditions \cite{An,MD,MalletParetSmith1990,Smillie1984, MalletParetSell1996a,MalletParetSell1996b,Matano,Matano2}. Precisely, the finite-codimensional transversality condition is:
\begin{itemize}
		\item
        There exists a closed linear subspace $\Pi\subset X$ of codimension $d$ such that, for any $1\le  i\le N$, any $x-y\in C^i$, and any $t>0$, if $\Phi_t(x)-\Phi_t(y)\in C^i\setminus C^{i-1}$, then $\Phi_t(x)-\Phi_t(y)\notin\Pi$.			
	\end{itemize}
The significance of the finite-codimensional transversality condition is already evident in the earlier work of Tere\v{s}\v{c}\'ak \cite{Tl}. Under the standing NICs hypotheses together with this condition, he proved that the $\omega$-limit set of every precompact orbit in an autonomous or time-periodic system admits a homeomorphic embedding into $\mathbb R^d$. In particular, when $d=1$, the $\omega$-limit set consists of a single equilibrium in the autonomous case and a single periodic orbit in the time-periodic case; when $d=2$, a Poincar\'e--Bendixson-type alternative holds. He further showed that the corresponding structures persist under sufficiently small $C^1$ perturbations. These results describe the asymptotic dynamics of individual orbits. 

In contrast, we establish a global counterpart for the Birkhoff center, which contains recurrent points arising from different orbits. More precisely, we consider the following finite-codimensional transversality condition, which is weaker than that considered by Tere\v{s}\v{c}ak \cite{Tl}:

\begin{enumerate}[\bf(H3)]
    \item There exists a closed linear subspace $\Pi\subset X$ of codimension $d$
such that, for any $1\le i\le N$ and any distinct $x, y\in \mathcal{B}(\Phi)$, if for any $s>0$, there exist $\tilde{x},\tilde{y}\in\mathcal{B}(\Phi)$
with $\Phi_s(\tilde{x})=x$ and $\Phi_s(\tilde{y})=y$ satisfying
$\Phi_t(\tilde{x})-\Phi_t(\tilde{y})\in
(C^i\setminus C^{i-1})\cup\{0\}$ for all $t\ge0$, then $x-y\notin\Pi$.
    \end{enumerate}
Under the finite-codimensional transversality condition {\rm(\hyperref[H-3]{H3})} and the assumptions of Theorem~\ref{Thm_2.3}, we prove that every connected component of the Birkhoff center admits a homeomorphic embedding into $\mathbb R^d$; equivalently, it is homeomorphic to a compact connected set in $\mathbb R^d$ endowed with the induced dynamics (see Theorem~\ref{Thm_2.7}). The dynamics on the component is therefore
topologically conjugate to the induced dynamics on its image. This provides a dimension-independent reduction of recurrent
dynamics: even when the original phase space is infinite-dimensional,
each connected recurrent component is topologically at most
$d$-dimensional.

Since the support of every invariant probability measure is contained
in the Birkhoff center, our main results immediately yield a
corresponding structural description of invariant measures. More
precisely, for every connected component $B$ of the support of an
invariant probability measure, there exists an index $j>0$ such that
$B$ is unordered with respect to $\{C^i\}_{i=0}^{j-1}$ and strongly
ordered with respect to $\{C^i\}_{i=j}^{N}$. Moreover, if the system
satisfies the finite-codimensional transversality condition {\rm(\hyperref[H-3]{H3})}, then
$B$ admits a homeomorphic embedding into $\mathbb R^d$; see
Theorem~\ref{Thm_2.10}. Thus, the measure-supported dynamics inherits
the same single-layer order structure and low-dimensional topological
constraints as the Birkhoff center.
If the codimension $d$ in {\rm(\hyperref[H-3]{H3})} is either $1$ or $2$ and the
restriction of the semiflow to the global attractor extends to a
continuous flow, then the low-dimensional embedding implies that the system has zero
topological entropy; see Theorem~\ref{Cor_2.4}. Notably, this
conclusion is independent of the dimension of the original phase
space and therefore applies even to infinite-dimensional systems.

We also apply our main results to two classes of systems:
bidirectional cyclic feedback systems and scalar parabolic equations
on the circle. For the former, an integer-valued Lyapunov function
gives rise to a finite family of NICs satisfying the required
hypotheses. For the latter, the natural infinite family of
zero-number NICs fails to satisfy uniform eventual strong
monotonicity, even for the linear heat equation. We demonstrate this
failure by an explicit counterexample and overcome it by constructing
a finite family of perturbed NICs. As a consequence, we obtain a
structural description of the Birkhoff centers and the supports of
invariant measures for these systems, and prove that both classes have
zero topological entropy; see Theorems~\ref{5-1} and~\ref{5-2}.

It is worth emphasizing that the Birkhoff center and the supports of
invariant measures encode the global organization of recurrent and
measure-supported dynamics, which cannot be fully recovered by
studying the limit set of an individual orbit. The main difficulty
is that the cone layer detected from a pair of recurrent points may
initially depend on that pair. To obtain a global description of a
connected component of the Birkhoff center, one must show that this
layer index is constant throughout the component and that the
component cannot cross the boundary between two adjacent cone layers.

To overcome this difficulty, we study the recurrent set as a whole
and establish a limit-set dichotomy for recurrent points
(Theorem~\ref{limit_set_dichotomy}). This dichotomy serves as the
main technical ingredient in the proofs of our structural results.
Its proof proceeds in three stages. First, we characterize the order
structure of the $\omega$-limit set associated with each recurrent
point (Subsection~\ref{OS}). Next, we examine how recurrent points
interact with the boundaries of the NICs
(Subsection~\ref{IRSaNB}). Finally, we extend this boundary analysis
to their $\omega$-limit sets (Subsection~\ref{iln}), thereby obtaining
the desired dichotomy.

The remainder of the paper is organized as follows. In
Section~\ref{s2}, we introduce the necessary notation and basic
definitions. In Section~\ref{S-3}, we formulate the standing
hypotheses and state the main results. Section~\ref{s3} establishes
the limit-set dichotomy for recurrent points, which is the main
technical ingredient in the proofs of our results. In
Section~\ref{S-5}, we analyze the structure of the Birkhoff center
and prove the main theorems. Section~\ref{S-6} applies the abstract
results to bidirectional cyclic feedback systems and scalar
parabolic equations on the circle. Finally,
Appendix~\ref{app:hausdorff} collects auxiliary results concerning
the Hausdorff distance and the separation index, while
Appendix~\ref{app:counterexample} presents a counterexample showing
that the natural zero-number NICs need not satisfy uniform eventual
strong monotonicity.

\section{Dynamical systems with NICs}\label{s2}
    In this section, we introduce dynamical systems with NICs. First, we present some basic notions of dynamical systems, including the Birkhoff center, in Subsection \ref{bn}. Then, in Subsection \ref{NICs}, we provide the definition of NICs and discuss some properties of systems with NICs.
    
	\subsection{Basic dynamical notions}\label{bn}
   
    Consider a \textit{continuous semiflow} $\Phi(t,x)$ (or simply $\Phi_t(x)$) on a Banach space $(X,\|\cdot \|_X)$, i.e., a continuous map $\Phi:\mathbb{R}^+\times X\rightarrow X$ satisfying $\Phi_{0}(x)=x$ and $\Phi_{t}(\Phi_{s}(x))=\Phi_{t+s}(x)$ for all $t,s\geq0$. A continuous semiflow $\Phi_t$ on $(X,\|\cdot \|_X)$ is said to be \textit{eventually compact} if there is $\tau > 0$ such that for any bounded set $B \subset X$, $\Phi_{\tau}(B)$ is relatively compact in $X$, where $\tau$ is called an eventually compact time of $\Phi_{t}$. Let $x\in X$, the \textit{positive orbit} of $x$ is denoted by $O^+(x) = \{\Phi_{t}(x):t\ge 0\}$; a \textit{negative orbit} of $x$ is denoted by $O^-(x) = \{z\in X: \Phi_{t}(z)=x \text{ for some } t>0\}$. It is well known that any point in an invariant set $A$ (i.e. $\Phi_{t}A=A$ for any $t \geq 0$) admits a negative orbit (see \cite[Section 2]{JKHALE}).

    The \textit{$\omega$-limit set} $\omega(x)$ of $x$ is defined by $\omega(x)=\bigcap_{t\geq0}\overline{\bigcup_{s\geq t}\Phi_{s}(x)}$, where the closure of a set $S\subset X$ is denoted by $\bar{S}$. A point $x\in X$ is called a \textit{recurrent point} if $x\in \omega(x)$.
	Let $\mathcal{R}(\Phi)$ denote the set of all recurrent points of $\Phi_t$ in $X$. The closure of $\mathcal{R}(\Phi)$ in $X$ is called the \textit{Birkhoff center}, denoted by $\mathcal{B}(\Phi)$, i.e.,
    $$\mathcal{B}(\Phi)=\overline{\left\{x\in X:x\in\omega(x)\right\}}.$$ 
    A \textit{global attractor} for continuous semiflow $\Phi_t$ is a compact invariant nonempty set $\Gamma \subset X$ satisfying $\lim_{t \to +\infty}\sup_{a \in \Phi_t(B)} \inf_{b\in \Gamma} \left \| a-b \right \|_X=0$ for every bounded set $B \subset X$. Clearly, if the continuous semiflow $\Phi_t$ has a global attractor, then the Birkhoff center $\mathcal{B}(\Phi)$ exists and is invariant.

   \subsection{Notions about NICs}\label{NICs}
   In this subsection, we introduce NICs. To this end, we first present the definition of $k$-cone, which is a key constituent of NICs.
   \begin{defn}
       A closed subset $C$ of $X$ is called a \textit{$k$-cone} of $X$ if it satisfies: 
       \begin{enumerate}[{\rm (i)}]
           \item for any $v\in C$ and $l\in \mathbb{R}$, it has $lv\in C$;
           \item $\max\left\{\operatorname{dim}W:W\subset C \text{ is a linear subspace}\right\}=k$.
       \end{enumerate}
   \end{defn}
    A \(k\)-cone \(C\) is called a \textit{solid \(k\)-cone} if $\operatorname{Int} C \ne \emptyset$. We explain the order ``$\simeq_C$" naturally induced by a $k$-cone $C$, that is, $x\simeq_C y$ whenever $y-x\in C$. We write $x\sim_{C} y$ if $y-x\in C\setminus\{0\}$ and $x\approx_{C} y$ if $y-x\in \text{Int}C$. A subset $S\subset X$ is \textit{unordered with $C$} if $S$ cannot contain distinct points related by $``\simeq_C"$; \textit{ordered with $C$} if any two distinct points in $S$ are related by $``\sim_C"$; \textit{strongly ordered with $C$} if any two  distinct points in $S$ are related by $``\approx_{C}"$.

	A continuous semiflow $\Phi_t$ is called \textit{monotone w.r.t \(k\)-cone $C$} if $x\sim_{C} y$ implies $\Phi_{t}(x)\simeq_{C}  \Phi_{t}(y)$ for any $t>0$; \textit{strongly monotone w.r.t solid \(k\)-cone $C$} if $x\sim_{C} y$ implies $\Phi_{t}(x)\approx_{C}  \Phi_{t}(y)$ for any $t>0$. A continuous semiflow $\Phi_t$ is called \textit{uniformly eventually strongly monotone} (UESM) w.r.t solid \(k\)-cone $C$ if there exists $T(C)>0$ such that $x\sim_{C} y$ implies $\Phi_{t}(x)\approx_{C}  \Phi_{t}(y)$ for any $t\ge  T(C)$, where $T(C)$ is called the \textit{UESM time w.r.t $C$}. 
	
	Now, we are going to present some notations about NICs and UESM semiflows w.r.t NICs.

    \begin{defn}\label{NIC}
        A sequence of closed sets $\{C^{i}\}_{i=0}^{N}$ ($N \in \mathbb { N }\cup \{+ \infty \}$) is called \textit{nested invariant cones} (NICs) for a continuous semiflow $\Phi_t$ on $X$, if the following conditions hold:
        \begin{enumerate}[{\rm (i)}]
        \item $\{ 0 \}=C ^ { 0 }\subset \dotsi \subset C ^ { i-1 } \subset C ^{ i }\subset \dotsi \subset C ^ { N } = X$;
        \item $C^i$ is a solid $k _ { i }$-cone for all $0 < i < N$, where $k _ { i-1 } < k _ { i }$ for all $1 < i < N$;
        \item  for each $1\leq i< N$, there exists $t_*^i\ge 0$ such that for any $x,y\in X$, if $x-y\in C^i$, then $\Phi_t(x)-\Phi_t(y)\in C^i$ for every $t\ge t_*^i$.
        \end{enumerate}
    \end{defn}

    \begin{defn}
        A continuous semiflow $\Phi_t$ on $X$ is said to be \textit{uniformly eventually strongly monotone} (UESM) w.r.t NICs $\{C^i\}_{i=0}^{N}$, if $\Phi_t$ is UESM w.r.t $C^i$ for every $1\le i<N$. 
	\end{defn}

    \begin{rmk}
    In the definition of NICs, the number of cones $N$ can be taken either finite or infinite values. Subsection \ref{finite} provides that a bidirectional cyclic feedback system is a system with NICs where $N$ is finite. Conversely, Subsection \ref{infinite} presents a parabolic equation on $S^{1}$ to illustrate NICs with infinite $N$, induced by the number of zeros in a function space.
    \end{rmk}

    \begin{rmk}
    The structure of NICs provides a hierarchical framework for analyzing system dynamics through a series of $k$-cones, which plays a crucial role in addressing complex dynamical problems. For instance, Fusco and Oliva \cite{Fusco1987,FOT} utilized NICs to address spatial decomposition and spectral problems, resolving transversality issues regarding invariant manifolds. For further details about NICs, see \cite{Tl, Margaliot-21,DZ}.
    \end{rmk}

	\section{Hypotheses and main results}\label{S-3}
    \subsection{Hypotheses}
	Throughout the paper, we impose the following standing hypotheses:
    
    \begin{enumerate}[\bf(H1)]
    \item The continuous semiflow $\Phi_t$ is eventually compact, and possesses a
    global attractor $\Gamma$ on a Banach space $(X,\|\cdot\|_X)$.\label{H-1}
    \end{enumerate}

    \begin{enumerate}[\bf(H2)]
    \item The continuous semiflow $\Phi_t$ is uniformly eventually strongly monotone (UESM) w.r.t nested  invariant cones (NICs) $\{C^i\}_{i=0}^{N}$.\label{H-2}
    \end{enumerate}

    \begin{enumerate}[\bf(H3)]
    \item There exists a closed linear subspace $\Pi\subset X$ of codimension $d$
such that, for any $1\le i\le N$ and any distinct $x,y\in\mathcal{B}(\Phi)$, if for any $s>0$, there exist $\tilde{x},\tilde{y}\in\mathcal{B}(\Phi)$
with $\Phi_s(\tilde{x})=x$ and $\Phi_s(\tilde{y})=y$ satisfying
$\Phi_t(\tilde{x})-\Phi_t(\tilde{y})\in (C^i\setminus C^{i-1})\cup\{0\}$ for all $t\ge0$, then $x-y\notin\Pi$.\label{H-3}
    \end{enumerate}

\begin{rmk}
In Tere\v{s}\v{c}ak's setting \cite{Tl}, the transversality condition requires that, for any \(1\le i\le N\), any \(x,y\in X\) with \(x-y\in C^i\), and every \(t>0\), if \(\Phi_t(x)-\Phi_t(y)\in C^i\setminus C^{i-1}\), then \(\Phi_t(x)-\Phi_t(y)\notin\Pi\). Hypothesis {\rm(\hyperref[H-3]{H3})} is weaker than this condition. Indeed, suppose that Tere\v{s}\v{c}ak's transversality condition holds and that \(x,y\in\mathcal{B}(\Phi)\) satisfy the premise of {\rm(\hyperref[H-3]{H3})}. For any \(s>0\), choose \(\widetilde{x},\widetilde{y}\in\mathcal{B}(\Phi)\) such that \(\Phi_s(\widetilde{x})=x\) and \(\Phi_s(\widetilde{y})=y\), with \(\widetilde{x}-\widetilde{y}\in C^i\) and \(x-y=\Phi_s(\widetilde{x})-\Phi_s(\widetilde{y})\in C^i\setminus C^{i-1}\). Tere\v{s}\v{c}ak's transversality condition therefore yields \(x-y\notin\Pi\). Hence Tere\v{s}\v{c}ak's condition implies {\rm(\hyperref[H-3]{H3})}, while the converse need not hold in general. In contrast, Hypothesis {\rm(\hyperref[H-3]{H3})} requires much less. It only needs to be checked for certain pairs $x,y\in\mathcal{B}(\Phi)$. For such pairs, their backward representatives stay in the same cone layer. Moreover, we only need to verify $x-y\notin\Pi$, rather than checking all pairs in $X$ at every positive time. 
\end{rmk}

\begin{rmk}
If the semiflow $\Phi_t$ extends to a flow on $\mathcal{B}(\Phi)$, then hypothesis {\rm(\hyperref[H-3]{H3})} admits a more transparent interpretation: there exists a closed linear subspace $\Pi\subset X$ of codimension $d$ such that, for any $1\le i\le N$ and any distinct points $x,y\in\mathcal{B}(\Phi)$, 
\begin{equation*}
    \forall t\in \mathbb{R},\ \Phi_t(x)-\Phi_t(y)\in C^i\setminus C^{i-1}
\quad\Longrightarrow\quad x-y\notin\Pi.
\end{equation*}

\end{rmk}

\subsection{Main results}
	
		The following theorems are the main results of our paper.
	\renewcommand{\thethmm}{\Alph{thmm}}	
	\begin{thmm}\label{Thm_2.3}
		Assume that {\rm(\hyperref[H-1]{H1})}-{\rm(\hyperref[H-2]{H2})} hold. Then, for any connected component $B$ of the Birkhoff center $\mathcal{B}(\Phi)$, there exists $j>0$ such that $B$ is
        \begin{enumerate}[{\rm (a)}]
        \item strongly ordered  with $\left \{ C^i \right \} _{i=j}^{N}$, i.e., strongly ordered with $C^i$ for every $i=j,\cdots,N$, and
        \item unordered with $\left \{ C^i \right \}_{i=0}^{j-1}$, i.e., unordered with $C^i$ for every $i=0,\cdots,j-1$.
        \end{enumerate}
        
    \end{thmm}

    \begin{rmk}
        Theorem \ref{Thm_2.3} presents feasible and effective information for the recurrence of systems with NICs. From a geometric perspective, NICs decompose $X$ as a combination of layers of NICs, i.e., $X\setminus\{0\}=\cup_{i=1}^{N}C^{i}\setminus C^{i-1}$. In this view, Theorem \ref{Thm_2.3} reveals a geometric structure of the Birkhoff center $\mathcal{B}(\Phi)$, showing that any connected component of $\mathcal{B}(\Phi)$ lies on a single layer of NICs. Since the Birkhoff center captures non-periodic recurrence which is recognized as the first indication of complicated asymptotic behavior, Theorem \ref{Thm_2.3} also demonstrates that, recurrent behaviors in systems with NICs possess a special order structure and thus complicated asymptotic behaviors can only happen in every layer of NICs.
    \end{rmk}

    \begin{thmm}\label{Thm_2.7}
        Assume that {\rm(\hyperref[H-1]{H1})}-{\rm(\hyperref[H-3]{H3})} hold. Then, every connected component $B$ of the Birkhoff center is homeomorphic to a compact invariant connected set in \(\mathbb{R}^d\). 
	\end{thmm}

    \begin{rmk}
        Assumption {\rm(\hyperref[H-3]{H3})} governs the dynamics in each layer of NICs by introducing a linear subspace of codimension $d$. Under {\rm(\hyperref[H-3]{H3})}, Theorem \ref{Thm_2.7} demonstrates that every connected component $B$ of the Birkhoff center is essentially $d$-dimensional, regardless of the dimension of $X$ (even if it is infinite). It is a reduction of topological dimension, which significantly simplifies the analysis of the asymptotic behavior of systems.
    \end{rmk}

    A direct consequence of Theorem \ref{Thm_2.3} and Theorem \ref{Thm_2.7} is a characterization of the structure of the supports of invariant measures. A Borel probability measure $\mu$ on $X$ is called invariant under $\Phi$ if $\mu\bigl(\Phi_t^{-1}(A)\bigr)=\mu(A)$ for every Borel set $A\subset X$ and every $t\geq 0$. The support of $\mu$, denoted by $\operatorname{supp}(\mu)$, is the complement of the union of all open sets of $\mu$-measure zero.
    Since $\text{supp}(\mu)\subset \mathcal{B}(\Phi)$ (see, e.g., \cite[p.28]{M12} or \cite[Appendix A]{J20}), Theorem \ref{Thm_2.3} and Theorem \ref{Thm_2.7} immediately imply the following result:
    \begin{thmm}\label{Thm_2.10}
		Assume that {\rm(\hyperref[H-1]{H1})}-{\rm(\hyperref[H-2]{H2})} hold. Then, for any connected component $B$ of the support of an invariant measure, there exists $j>0$ such that $B$ is strongly ordered  with $\left \{ C^i \right \} _{i=j}^{N}$ and unordered with $\left \{ C^i \right \}_{i=0}^{j-1}$. 
        
        Moreover, if {\rm(\hyperref[H-3]{H3})} holds, $B$ is homeomorphic to a compact invariant connected set in \(\mathbb{R}^d\).
    \end{thmm}

     In particular, when the codimension $d$ in {\rm(\hyperref[H-3]{H3})} is either $1$ or $2$, it follows from Theorem \ref{Thm_2.10} that the topological entropy of such systems is zero. Precisely, we obtain the following theorem:
    
	\begin{thmm}\label{Cor_2.4}
    Assume that {\rm(\hyperref[H-1]{H1})}-{\rm(\hyperref[H-3]{H3})} hold and $\Phi_t$ can be extended to a flow on $\Gamma$. If $d=1,2$ in {\rm(\hyperref[H-3]{H3})}, then the topological entropy of $\Phi_t$ is $0$.
	\end{thmm}

   \begin{rmk}
       Theorem \ref{Cor_2.4} reveals the connection between dynamics in layers of NICs and the dynamical complexity of systems. Specifically, for a system with NICs, if the codimension of the linear subspace in {\rm(\hyperref[H-3]{H3})} is low (i.e., $d=1,2$), then its topological entropy is zero and thus there is no chaotic behavior in it. In general, the topological entropy for high-dimensional systems is difficult to calculate. However, Theorem \ref{Cor_2.4} is independent of the dimension of the ambient phase space and therefore applies even to infinite-dimensional systems. From a theoretical perspective, our result provides a criterion for ruling out chaos in infinite-dimensional systems, and demonstrates that low-dimensional nested cone structures can suppress dynamical complexity in essence.
   \end{rmk}

	\noindent\section{Limit-set dichotomy of $\mathcal{R}(\Phi)$}\label{s3}
	In this section, we study the asymptotic behavior of any recurrent point and present a technical theorem, called \textit{Limit-set dichotomy of $\mathcal{R}(\Phi)$}, which turns out to be crucial to the proof of Theorem \ref{Thm_2.3}.

    Throughout this section, we assume that {\rm(\hyperref[H-1]{H1})} and {\rm(\hyperref[H-2]{H2})} hold. For simplicity, we label $T_*^i=\operatorname{max}\left \{\tau,t_*^i,T(C^i)\right\}$ for any $1\leq i< N$, where $\tau$ is eventually compact time, $t_*^i$ is given in Definition \ref{NIC} and $T(C^i)$ is UESM time w.r.t the cone $C^i$.

	\begin{thm}\label{limit_set_dichotomy}{\rm(Limit-set dichotomy of $\mathcal{R}(\Phi)$).}
	Let $x,y\in \mathcal{R}(\Phi)$, then there exists $j>0$ such that
      \begin{enumerate}[{\rm (a)}]
          \item $\omega(x)$ and $\omega(y)$ are strongly ordered with $\left \{ C^i \right \} _{i=j}^{N}$, and
          \item $\omega(x)$ and $\omega(y)$ are unordered with $\left \{ C^i \right \} _{i=0}^{j-1}$. 
      \end{enumerate}
	\end{thm}

    Here, two sets $U$, $V$ are said to be \textit{strongly ordered with $\{C_{i}\}_{i\in\mathcal{I}}$}, if for any $u\in U$ and $v\in V$ it has $u-v\in \text{Int}C^i$ for any $i\in \mathcal{I}$; \textit{unordered with $\{C_{i}\}_{i\in \mathcal{I}}$}, if there are no points $u\in U$ and $v\in V$ related by $``\simeq_{C^i}"$ for any $i\in \mathcal{I}$, where $\mathcal{I}\subset \mathbb{N} \cup \{+\infty\}$ is an index set of the NICs. 

	\begin{rmk}
    In classical monotone systems, the invariant cone is convex which generates a partial order ``$<$". Hirsch \cite[Theorem 0.8]{H88-S} presented the Limit-set dichotomy for any two ordered points, that is, for any two points $x,y$ with $x < y$, it has either $\omega(x) < \omega(y)$ or $\omega(x) = \omega(y)$. However, for systems with NICs, each cone generally is not convex. This theorem establishes the Limit-set dichotomy of $\mathcal{R}(\Phi)$; while it does not characterize the asymptotic behavior of every point, it is sufficient to capture the global recurrent behaviors of the system.
	\end{rmk}

    The proof of Theorem \ref{limit_set_dichotomy} is presented in Subsection \ref{ptl}. For this, we first study the order structure of limit sets of any recurrent point (Proposition \ref{OR}) in Subsection \ref{OS}. Then, we present intersection of recurrent set and boundaries of NICs (Lemma \ref{RC}) in Subsection \ref{IRSaNB}, and intersection of limit sets and boundaries of NICs (Lemma \ref{cor}) in Subsection \ref{iln}.

	\subsection{Order structure of $\omega$-limit sets}\label{OS}

    In this subsection, we consider the order structure of limit sets for any recurrent point, which plays an important role in the proof of Theorem \ref{limit_set_dichotomy}.
	
	\begin{prop}\label{OR}
	   For any point $x\in \mathcal{R}\left(\Phi\right)$, there exists $j>0$ such that 
       \begin{enumerate}[{\rm (a)}]
          \item $\omega\left(x\right)$ is strongly ordered with $\left \{ C^i \right \} _{i=j}^{N}$, and
          \item $\omega\left(x\right)$ is unordered with $\left \{ C^i \right \} _{i=0}^{j-1}$. 
      \end{enumerate}
	\end{prop}
	\begin{proof}
    If $x$ is an equilibrium, then the proposition is obtained by $j=N$. In the remainder, we assume that $x$ is not an equilibrium. 
        
        We claim that \textit{$\omega(x)$ is either unordered or strongly ordered with  $C^i$ for any $0<i<N$.} Before we prove the claim, we show how to obtain the proposition by the claim. If $\omega(x)$ is unordered with $C^i$ for all $0<i<N$, the proposition follows by taking $j=N$. Otherwise, there exists $0<j<N$ such that $\omega(x)$ is strongly ordered with $C_j$ and unordered with $C_{j-1}$. Given the nest-cone property $C^j\subset C^{i}$ for any $i>j$, $\omega(x)$ is strongly ordered with $\left \{ C^i \right \} _{i=j}^{N}$. Likewise, since $C^i\subset C^{j}$ for any $i<j$, $\omega(x)$ is unordered with $\left \{ C^i \right \} _{i=0}^{j-1}$. This concludes the proof.

        Now, it remains to prove the claim.

        As a first step, we show if there is $\tau>0$ such that $x \sim_{C^i} \Phi_{\tau}(x)$, then $\omega(x)$ strongly ordered with  $C^i$. Let $T_1 = \inf \left\{ t > 0 : x \simeq_{C^i} \Phi_t(x) \right\}$. Clearly, $x \simeq_{C^i} \Phi_{T_1}(x)$ and $\tau\geq T_1$. We prove $T_1=0$ by contradiction. Suppose $T_1 > 0$. Then $\tau=nT_1 + s$, where \(n \in \mathbb{N}\), \(s \in (0, T_1)\). If $x = \Phi_{T_1}(x)$, one has $\Phi_{\tau}(x) =\Phi_{nT_1 + s}(x) =\Phi_{s}(x)$, and hence $x \sim_{C^i}\Phi_{s}(x)$ as \(x \sim_{C^i} \Phi_{\tau}(x)\) which implies \(s \in \{ t > 0 : x \sim_C \Phi_t(x) \}\) contradicting the definition of \(T_1\). 
        Otherwise, \(x \sim \Phi_{T_1}(x)\). By {\rm(\hyperref[H-2]{H2})}, it has \(\Phi_{T^i_*}(x) \approx \Phi_{T^i_*}(\Phi_{T_1}(x))\).
        Thus, there exists a neighborhood $U$ of $\Phi_{T_1}(x)$, such that $\Phi_t(x) \approx_{C^i} \Phi_t(U)$ for any  $t \ge T^i_*$.
		Recalling $x$ is a recurrent point, there exists a sequence $\{t_n\}_{n \in \mathbb{N}}\rightarrow\infty$ such that $\Phi_{t_n}(x)\rightarrow x$ as $n\rightarrow\infty$. Choose sufficiently small $a \in (0, T_1)$ such that $\Phi_{T_1 - a}(x)\neq x$ and $\Phi_{T_1-a}(x)\in U$. Recalling $\Phi_t(x) \approx_{C^i} \Phi_t(U)$ for all  $t \ge T^i_*$, it has $\Phi_{t_n}(x) \approx_{C^i} \Phi_{t_n}(\Phi_{T_1-a}(x))$ for any $t_n \ge T^i_*$.
		Letting \(n \to \infty\) and by the closeness of $C^i$, we obtain  
		$x \simeq_{C^i} \Phi_{T_1 - a}(x)$. Noting $\Phi_{T_1 - a}(x)\neq x$, it has $x \sim_{C^i} \Phi_{T_1 - a}(x)$ which implies \(T_1 - a \in \{ t > 0 : x \simeq_C \Phi_t(x) \}\)
		contradicting the definition of \(T_1\). Thus \(T_1 = 0\). Let $T_2= \sup \left\{ m \ge 0 : x \simeq_C \Phi_s(x) \text{ for all } s \in [0, m] \right\}$.
		By an argument similar to that for \(T_1\), we obtain \(T_2 = +\infty\).
        Thus, $x \simeq_{C^i} \Phi_{t}(x)$ for any $t\ge 0$. Again, since $x\in \mathcal{R}(\Phi)$, there exists a sequence $\{t_n\}_{n \in \mathbb{N}}$ with  $t_n\to +\infty$  as $n\to +\infty$ such that $\Phi_{t_n}(x)$ converges to $x$. Recalling that $x \simeq_{C^i} \Phi_{t}(x)$ for any $t\ge 0$, it has $\Phi_{t_n + \tau_1}(x) \simeq_{C^i} \Phi_{t_n + \tau_2}(x)$ for any $\tau_2>\tau_1>0$. 
        Thus by the closeness of $C^i$, one has $\Phi_{t_2}(x) \simeq_{C^i} \Phi_{t_1}(x)$. 
        Therefore, $O^+(x)$ is an ordered set with $C^i$, and hence, $\overline{O^+(x)}$ is an ordered set with $C^i$. Noting $x\in \mathcal{R}(\Phi)$, it follows that $\omega(x)$ is an ordered set with $C^i$. 
        Moreover, as $\omega(x)$ is invariant, by {\rm(\hyperref[H-2]{H2})}, we obtain $\omega(x)$ is strongly ordered with $C^i$.

       As a second step, we show if $x \nsim_{C^i} \Phi_{t}(x)$  for all $ t > 0$, then $\omega(x)$ unordered  with  $C^i$. Suppose not, that is, there are two points \(a, b \in \omega(x)\) with \(a \sim_{C^i} b\). By {\rm(\hyperref[H-2]{H2})}, it has \(\Phi_{T^i_*}(a) \approx_{C^i}\Phi_{T^i_*}(b)\). Thus there exist neighborhoods \(U\) of \(\Phi_{T^i_*}(a)\) and \(V\) of \(\Phi_{T^i_*}(b)\) such that $U \approx_{C^i}V$.
        Choose two points $\Phi_{s_1}(x), \Phi_{s_2}(x)$ in $O^+(x)$ with $s_2>s_1>T^i_*$ such that \(\Phi_{s_1}(x) \in U\), \(\Phi_{s_2}(x) \in V\). Clearly, it has \(\Phi_{s_1}(x) \approx_{C^i}\Phi_{s_2}(x)\) as $U \approx_{C^i}V$. Since $x$ is a recurrent point, there exists a sequence $\left\{t_n\right\}_{n\in \mathbb{N}}$ with  $t_n\to +\infty$  as $n\to +\infty$ such that $\Phi_{t_n+s_1}(x)$ converges to $x$. Moreover, it has $\Phi_{t_n+s_2}(x)$ converges to $\Phi_{s_2-s_1}(x)$. Recalling \(\Phi_{s_1}(x) \approx_{C^i}\Phi_{s_2}(x)\), one has \(\Phi_{t_n+s_1}(x) \approx_{C^i}\Phi_{t_n+s_2}(x)\), i.e., $\Phi_{t_n+s_1}(x) -\Phi_{t_n+s_2}(x)\in \operatorname{Int}C^i$, for any $n\in \mathbb{N}$. Thus, by the closeness of $C^i$, we have \(x-\Phi_{s_2-s_1}(x)\in C^i\), that is, $x\simeq\Phi_{s_2-s_1}(x)$, which contradicts \(x \nsim_{C^i} \Phi_{t}(x)\) for all \(t > 0\). Thus, $\omega\left(x\right)$ is unordered with $C^i$.

       As a third step, we show if $x$ is a periodic point with minimal positive period $P$ and \(x \nsim_{C^i} \Phi_{t}(x)\) for all \(t\in \left(0,P\right)\), then $\omega(x)$ unordered  with  $C^i$. Suppose not, that is, there are two points \(a, b \in \omega(x)\) with \(a \sim_{C^i} b\). By {\rm(\hyperref[H-2]{H2})}, it has \(\Phi_{T^i_*}(a) \approx_{C^i}\Phi_{T^i_*}(b)\), and hence, there exist neighborhoods \(U\) of \(\Phi_{T^i_*}(a)\) and \(V\) of \(\Phi_{T^i_*}(b)\) such that $U \approx_{C^i}V$.
        Choose two points $\Phi_{s_1}(x), \Phi_{s_2}(x)$ in $O^+(x)$ with $s_2>s_1>T^i_*$ and $s_2 - s_1 \notin \{ nP : n \in \mathbb{Z}^+ \}$ such that \(\Phi_{s_1}(x) \in U\), \(\Phi_{s_2}(x) \in V\). Clearly, it has \(\Phi_{s_1}(x) \sim_{C^i}\Phi_{s_2}(x)\). Since $x$ is a recurrent point, there exists a sequence $\left\{t_n\right\}_{n\in \mathbb{N}}$ with  $t_n\to +\infty$  as $n\to +\infty$ such that $\Phi_{t_n+s_1}(x)$ converges to $x$. Moreover, it has $\Phi_{t_n+s_2}(x)$ converges to $\Phi_{s_2-s_1}(x)$. Recalling \(\Phi_{s_1}(x) \sim_{C^i}\Phi_{s_2}(x)\), one has \(\Phi_{t_n+s_1}(x)\sim _{C^i}\Phi_{t_n+s_2}(x)\), i.e., $\Phi_{t_n+s_1}(x) -\Phi_{t_n+s_2}(x)\in C^i$, for any $n\in \mathbb{N}$. Thus, by the closeness of $C^i$, we have \(x-\Phi_{s_2-s_1}(x)\in C^i\), that is $x\sim_{C^i}\Phi_{s_2-s_1}(x)$. Since $s_2-s_1=nP+ s$ where \(n \in \mathbb{N}\) and \(s \in (0, P)\), it has $\Phi_{s_2-s_1}(x)=\Phi_s(x)$. Hence, $x\sim_{C^i}\Phi_{s}(x)$ for some $s\in (0,P)$, contradicting \(x \nsim_{C^i} \Phi_{t}(x)\) for all \(t \in (0,P)\). Thus, $\omega\left(x\right)$ is unordered with $C^i$.

        This establishes the claim and thus we have completed the proof. 
	\end{proof}

    \subsection{Recurrence in the boundary of NICs}\label{IRSaNB}

    In this subsection, we study the recurrence in the boundary of NICs, and show that for any recurrent point $x$, $x$ is the only recurrent point in its boundary of NICs. 
    
    To express clearly, we label some useful notation. For any $x\in X$ and cone $C$, we define a translated cone $C_x=x+C$ and the boundary of $C_x$ is labeled by $\partial C_x$. For any $x\in X$ and NICs $\{C^i\}_{i=0}^{N}$, the boundary of translated NICs $\{C^i\}_{i=0}^{N}$ is denoted by $\bigcup_{i=0}^{N} \partial C^i_x $. 
    
    Now, we present the lemma:

    \begin{lem}\label{RC}
		Let $x\in \mathcal{R}(\Phi)$, then $\mathcal{R}(\Phi)\cap (\bigcup_{i=0}^{N} \partial C^i_x)=\{x\}$.
	\end{lem}

    To prove this lemma, we present some important propositions to describe the dynamical behaviors on the boundary of NICs. To this end, we introduce the Hausdorff distance $d_H$ and the separation index $\underline{dist}$, with detailed definitions and properties given in the Appendix \ref{app:hausdorff}.

    \begin{prop}\label{R-dist}
		Let $x,z\in \mathcal{R}(\Phi)$ with $z\in \bigcup_{i=0}^{N} \partial C^i_x$ and $x\ne z$. Then there exist $0<j<N$, $t_0 \geq 2T^{j}_*$ and $\eta > 0$ such that $$\underline{dist}(\Phi_{t_0} (\omega(z)\cap C^j_x), \partial C_x^j) \ge  \eta.$$
	\end{prop}
	
	\begin{proof}
    Clearly, there is $0<j<N$ such that $z\in \partial C^j_x$, and hence, $\omega(z)\cap C^j_x\neq\emptyset$ (since $z\in \omega(z)$). For any $y\in X$ and $t \geq 2T^{j}_*$, we label $$\Omega_y^j=\omega(z)\cap C^j_y,\ \ \Omega^j_x(t) = x - \Phi_t(x) + \Phi_t \Omega^j_x.$$ Noting that $x\notin \omega(z)$ by Proposition  \ref{OR}, one can find $\eta_2>\eta_1>0$ such that $\omega(z)\subset\operatorname{Int}R$, where $R = \{ v \in X : \eta_1 \leq \|v - x\|_X \leq \eta_2 \}$.

     We claim there is $\delta>0$ such that, for any $t \geq 2T^{j}_*$, it has 
        \begin{equation}\label{1}
            \underline{\mathrm{dist}}(\Omega^j_x(t), \partial C^j_x \cap R) > \delta.
        \end{equation}
    Before proving the claim, we show how it implies the proposition. Take 
		$$\eta = \min\left\{  \underline{\mathrm{dist}}\left(\omega(z), \partial C^j_x \setminus\operatorname{Int} R\right) ,\frac{\delta}{2}\right\}.$$ 
        Clearly, $\eta>0$ as $\omega(z)\in \operatorname{Int} R$. Since $x\in \mathcal{R}(\Phi )$, there is $t_0\ge  2T^{j}_*$ such that $\|x - \Phi_{t_0}(x)\|_X\leq \eta$, and hence, $d_H(\Phi_{t_0} \Omega^j_x, \Omega^j_x(t_0)) \leq \eta$. 
		By (\ref{dis_2}), we have  
		  \begin{equation}\label{eta-R}
		      \underline{\mathrm{dist}}(\Phi_{t_0} \Omega^j_x, \partial C^j_x \cap R) \geq \underline{\mathrm{dist}}(\Omega^j_x(t_0), \partial C^j_x \cap R) - d_H(\Phi_{t_0} \Omega^j_x, \Omega^j_x(t_0)) \geq \delta - \eta \geq \eta.
		  \end{equation} 
        On the other hand, by the definition of $\eta$ and the fact that $\Omega^j_x\subset \omega(z)$, it has 
		\begin{equation}\label{disteta}
			\underline{\mathrm{dist}}(\Phi_{t_0} \Omega^j_x, \partial C^j_x \setminus\operatorname{Int} R) \geq  \eta.
		\end{equation}
        Thus, together with (\ref{eta-R}) and (\ref{disteta}), it has $\underline{\mathrm{dist}}(\Phi_{t_0} \Omega^j_x, \partial C^j_x ) \geq  \eta$. This completes the proof of Proposition \ref{R-dist}.

       Now, we proceed to prove the claim. First, we assert for any $t\geq 2T^{j}_*$, there exists a compact set $D_x(t)$ such that $\Omega^j_x(t)\subseteq D_x(t)$ and $D_x(t)\cap (\partial C^j_x \cap R)=\emptyset$. Let $B= C^j\cap (\omega(z)-\omega(x))$ and $B_y=y+B$ for any $y \in \overline{O^+(x)}$. Clearly, $B$ is nonempty since $z-x\in B$. Then by the compactness of $\omega(z)$ and $\omega(x)$, $B$ is compact. Let $$
		D_x(t) = x - \Phi_t(x) + \Phi_{T^{j}_*} B_{\Phi_{t-T^{j}_*}(x)} \text{ for any }t\geq T^{j}_*.$$
       Clearly, $D_x(t)$ is compact for any $t\geq 2T^{j}_*$. 
       Noting a fact that $\Omega_y^j\subset B_y$ for any $y \in \overline{O^+(x)}$, it has for any $t\ge 2T^{j}_*$, $$\Omega^j_x(t) \subset x - \Phi_t(x) + \Phi_{T^{j}_*} (\Omega^j_{\Phi_{t-T^{j}_*}(x)})\subset x - \Phi_t(x) + \Phi_{T^{j}_*}( B_{\Phi_{t-T^{j}_*}(x)})=D_x(t).$$ Thus, it has $\Omega^j_x(t) \subset D_x(t)$ for any $t\ge 2T^{j}_*$, which is the first part of assertion. Noting $ B_{\Phi_{t-T^{j}_*}(x)}\subset C_{\Phi_{t-T^{j}_*}(x)}^j$, it has $\Phi_{T^{j}_*}(B_{\Phi_{t-T^{j}_*}(x)})\subset \operatorname{Int}C_{\Phi_{t}(x)}^j\cup\{\Phi_{t}(x)\}$, and hence $D_x(t)\subset \operatorname{Int}C_x^j\cup \{x\}$ for any $t\ge 2T^{j}_*$. Therefore, it has $D_x(t)\cap (\partial C^j_x \cap R)=\emptyset$ for any $t\ge 2T^{j}_*$. This completes the proof of the assertion.
        
        Then, by the compactness of $D_x(t)$ and the closeness of $\partial C^j_x \cap R$, one has $\underline{\mathrm{dist}}(D_x(t), \partial C^j_x \cap R) > 0$ for any $t\ge 2T^{j}_*$. By virtue of Lemma \ref{A-dist}, $\underline{\mathrm{dist}}(D_x(t), \partial C^j_x \cap R)$ is continuous with $t$. Thus, by the compactness of $\overline{O^+(x)}$, there exists $\delta > 0$ such that  $\underline{\mathrm{dist}}(D_x(t), \partial C^j_x \cap R) > \delta$ for any $t\ge 2T^{j}_*$.
        Recalling $\Omega^j_x(t)\subset D_x(t)$ for any $t\ge 2T^{j}_*$, we obtain \eqref{1} and complete the proof of claim.  
    \end{proof}

        \begin{cor}\label{R-dist1}
        Let $x,z\in \mathcal{R}(\Phi)$ with $z\in \bigcup_{i=0}^{N} \partial C^i_x$ and $x\ne z$. Then $$\Phi_{t_0} (\omega(z)\cap C^j_x) \subseteq \omega(z)\cap C^j_x,$$ where $j, t_0$ are given in Proposition \ref{R-dist}.
        \end{cor}

        \begin{proof}
            Suppose on the contrary that there is $y\in \Phi_{t_0} (\omega(z)\cap C^j_x))\setminus(\omega(z)\cap C^j_x)$. Then it follows from proposition \ref{R-dist} that $\underline{\mathrm{dist}}(y, \partial C^j_x) > \eta$, and hence, $\underline{\mathrm{dist}}(y, C^j_x)> \eta$ (since $y \notin C^j_x$). We define $\Omega^j_x(t)$ exactly as in Proposition \ref{R-dist}, i.e., $\Omega^j_x(t) = x - \Phi_t(x) + \Phi_t (\omega(z)\cap C^j_x)$. Recalling $\Omega^j_x(t_0) \subseteq C^j_x$, one has 
            \begin{equation}\label{>}
                \underline{\mathrm{dist}}(y, \Omega^j_x(t_0)) > \eta.
            \end{equation}
		On the other hand, noting $y \in \Phi_{t_0}(\omega(z)\cap C^j_x)$ and $d_H(\Phi_{t_0} (\omega(z)\cap C^j_x), \Omega^j_x(t_0)) \leq \eta$, it has 
        \begin{equation*}
            \underline{\mathrm{dist}}(y, \Omega^j_x(t_0)) \leq \eta,
        \end{equation*}
		which contradicts inequality (\ref{>}). Thus, $\Phi_{t_0} (\omega(z)\cap C^j_x) \subseteq \omega(z)\cap C^j_x$.
        \end{proof}

        \begin{prop}\label{IP}
		Let $z\in \mathcal{R}(\Phi)$. Then for any $\theta > 0$, $\varepsilon > 0$ and $\tau>0$, it has
		$$N(z, \theta) \cap \mathcal{T}(\tau, \varepsilon) \neq \emptyset,$$
		where $N(z, \theta) = \{ t > 0 : |\Phi_t(z) - z| < \theta \}$ and $\mathcal{T}(\tau, \varepsilon) = \{ n\tau + t : n \in \mathbb{Z}^+, |t| < \varepsilon \}$.
	\end{prop}
	\begin{proof}
		See \cite[Appendix]{SWZ}.
	\end{proof}

    Now, we are ready to prove the Lemma \ref{RC}.

        \begin{proof}[Proof of Lemma \ref{RC}]
		We prove it by contradiction. Suppose there is $z\in \mathcal{R}(\Phi)$ with $z\neq x$ such that $z\in\bigcup_{i=0}^{N} \partial C^i_x$. By Proposition \ref{R-dist}, there exist $0<j<N$, $\eta>0$ and $t_0\ge 
        2T^{j}_*$ such that \begin{equation}\label{q1}
		    \underline{\mathrm{dist}}(\Phi_{t_0}(\Omega_x^j),  \partial C^j_x)\ge \eta.
		\end{equation}
        By Lemma \ref{AA-dist}, there exists $\varepsilon_0 > 0$ such that for any $|t| < \varepsilon_0$,
		\begin{equation}\label{q2}
		    d_H(\Phi_{t_0 + t}(\Omega_x^j), \Phi_{t_0}(\Omega_x^j)) < \frac{\eta}{2}.
		\end{equation}
        By (\ref{dis_2}), one has for any $t \in \mathbb{R}$,
        \begin{equation}\label{q3}
            \underline{\mathrm{dist}}(\Phi_{t_0 + t}(\Omega_x^j),  \partial C^j_x) \geq \underline{\mathrm{dist}}(\Phi_{t_0}(\Omega_x^j),  \partial C^j_x) - d_H(\Phi_{t_0 + t}(\Omega_x^j), \Phi_{t_0}(\Omega_x^j)) .
        \end{equation}
        Together \eqref{q1},\eqref{q2} with \eqref{q3}, one has for any $|t| < \varepsilon_0$,
        \begin{equation}\label{q4}
            \underline{\mathrm{dist}}(\Phi_{t_0 + t}(\Omega_x^j),  \partial C^j_x)>\frac{\eta}{2}.
        \end{equation}
        Denote the time-set by $\mathcal{T}(t_0, \varepsilon_0) \triangleq \{ nt_0 + t : n \in \mathbb{Z}^+, |t| < \varepsilon_0 \}$. Furthermore, by Corollary \ref{R-dist1}, \eqref{q4} can be enhanced as 
        \begin{equation}\label{q5}
            \underline{\mathrm{dist}}(\Phi_{s}(\Omega_x^j),  \partial C^j_x)>\frac{\eta}{2}, \text{ for any } s\in \mathcal{T}(t_0, \varepsilon_0).
        \end{equation}
        Recalling $z\in \Omega_x^j$ and $z\in \partial C^j_x$, by \eqref{q5}, one has
        \begin{equation*}\label{q6}
            \underline{\mathrm{dist}}(\Phi_{s}(z),  z)>\frac{\eta}{2}, \text{ for any } s\in \mathcal{T}(t_0, \varepsilon_0).
        \end{equation*}
        Thus, 
		\[
		\left\{t > 0 : \|\Phi_t(z) - z\|_X < \frac{\eta}{2}\right\}\cap \mathcal{T}(t_0, \varepsilon_0) = \emptyset,
		\]
		which contradicts Proposition \ref{IP}.
	\end{proof}

    \subsection{Intersection of $\omega$-limit Sets and boundary of NICs}\label{iln}

    In this section, we characterize the intersection of $\omega$-limit sets of recurrent points and the boundary of NICs.

    \begin{lem}\label{cor}
		 Let $x,y\in \mathcal{R}(\Phi)$, then
		$\omega(y)\cap((\bigcup_{i=0}^{N} \partial C^i_x)\setminus\{x\})=\emptyset$.
	\end{lem}
\begin{proof}
    We consider it in the following two cases: (i) $x\in \omega(y)$; (ii) $x\notin \omega(y)$. 
    
    (i) $x\in \omega(y)$. For any $0\leq i\leq N$, Proposition \ref{OR} implies that $\omega(y)$ is either strongly ordered or unordered with $C^{i}$. Consequently, $\omega(y)\setminus\{x\}$ and $\{x\}$ are strongly ordered or unordered with $C^{i}$. Thus, $\omega(y)\setminus\{x\}\cap(\bigcup_{i=0}^{N} \partial C^i_x)=\emptyset$, that is, $\omega(y)\cap((\bigcup_{i=0}^{N} \partial C^i_x)\setminus\{x\})=\emptyset$.

    (ii) $x\notin \omega(y)$. First, we claim that there is $j>0$ such that
    \begin{equation}\label{cjj}
        O^+(y)\subset( \operatorname{Int}C^j_x)\setminus C^{j-1}_x.
    \end{equation}
    In fact, by the virtue of Lemma \ref{RC}, together with $y\in \mathcal{R}(\Phi)$, it has $O^+(y)\cap ((\bigcup_{i=0}^{N} \partial C^i_x)\setminus\{x\})=\emptyset$. Thus, $O^+(y)\cap\bigcup_{i=1}^{N} (\left (\operatorname{Int} C^i_x\right )\setminus C^{i-1}_x))\neq\emptyset$, and hence, there is $j>0$ such that $O^+(y)\cap( ( \operatorname{Int}C^j_x)\setminus C^{j-1}_x)\neq\emptyset$. By the virtue of Lemma \ref{RC} and the connectedness of $O^+(y)$, it has $O^+(y)\subset( \operatorname{Int}C^j_x)\setminus C^{j-1}_x$, which completes the proof of claim.

    Then, we prove the lemma by contradiction. Suppose $z\in \omega(y) \cap ((\bigcup_{i=0}^{N} \partial C^i_x)\setminus\{x\})$. 
    By the claim, it has $\omega(y)\subset C^j_x\setminus \text{Int}C^{j-1}_x$, and hence, $z\in \partial C^{j-1}_x\cup\partial C^j_x$. We consider the following two cases: (iia) $z\in \partial C^{j-1}_x$; (iib) $z\in \partial C^{j}_x$.

    (iia) $z\in \partial C^{j-1}_x$. Since $y\in ( \operatorname{Int}C^j_x) \setminus C^{j-1}_x$, there are neighborhoods $U_x$ of $x$ and $U_y$ of $y$ such that $y'-x'\in ( \operatorname{Int}C^j) \setminus C^{j-1}$ for any $x'\in U_x$ and $y'\in U_y$. Recalling $x$, $y\in R\left(\Phi\right)$, there are $\tau_x,\tau_y>T_*^{j-1}$ such that $\Phi_{\tau_x}(x)\in U_x$ and $\Phi_{\tau_y}(y)\in U_y$, and hence, 
    \begin{equation}\label{in1}
        \Phi_{\tau_y}(y) - \Phi_{\tau_x}(x) \in ( \operatorname{Int}C^j)\setminus C^{j-1}.
    \end{equation}
	Moreover, since $z-x\in \partial C^{j-1}$, it has $\Phi_{\tau_x}(z) - \Phi_{\tau_x}(x) \in \operatorname{Int}C^{j-1}$ by {\rm(\hyperref[H-2]{H2})}. Recalling $z\in \omega(y)$, there is $y_1\in O^+(y)$ close to $\Phi_{\tau_x}(z)$ such that 
    \begin{equation}\label{in2}
        y_1 - \Phi_{\tau_x}(x) \in \operatorname{Int}C^{j-1}.
    \end{equation}
    By the connectedness of $O^+(y)$, together with \eqref{in1} and \eqref{in2}, there is $y_2 \in O^+(y)$ such that $y_2-\Phi_{\tau_x}(x)\in \partial C^{j-1}$. On the other hand, together with \eqref{in1} and \eqref{in2}, it follows from Proposition \ref{OR} that $\Phi_{\tau_x}(x)\notin O^+(y)$. Thus, $y_2\in \partial C^{j-1}_{\Phi_{\tau_x}(x)}\setminus\{\Phi_{\tau_x}(x)\}$, which contradicts Lemma \ref{RC}.

    (iib) $z\in \partial C^{j}_x$. 
    Fix $t_0>T^j_*$. 
By the invariance of $\omega(x)$ and $\omega(y)$, choose 
$x_0\in\omega(x)$ and $y_0,z_0\in\omega(y)$ such that $\Phi_{t_0}(x_0)=x$,
 $\Phi_{t_0}(y_0)=y$ and $
\Phi_{t_0}(z_0)=z$. Clearly, $x_0,y_0,z_0\in \mathcal{R}(\Phi)$. Moreover, it has $z_0-x_0\notin C^j$ by the fact $z-x\in \partial C^j$ and {\rm(\hyperref[H-2]{H2})}.

Noting that $y_0-x_0\notin\partial C^j$ by Lemma \ref{RC} as $x_0,y_0\in \mathcal{R}(\Phi)$, we consider the following two cases: (1) $y_0-x_0\in\operatorname{Int} C^j$; (2) $y_0-x_0\notin C^j$.

(1) $y_0-x_0\in\operatorname{Int} C^j$. Then there is $y_1\in O^+(y)$ close 
to $y_0$ such that 
\begin{equation}\label{y1}
    y_1-x_0\in\operatorname{Int} C^j.
\end{equation}
Since $z_0-x_0\notin C^j$, there is $z_1\in O^+(y)$ such that 
\begin{equation}\label{z1}
    z_1-x_0\notin C^j.
\end{equation}
Thus, together with \eqref{y1} and \eqref{z1}, it has $x_0\notin O^+(y)$ by Proposition \ref{OR}. Hence, by the connectedness of $O^+(y)$, together with \eqref{y1} and \eqref{z1}, there exists $y_2\in O^+(y)\cap
\bigl(\partial C_{x_0}^j\setminus\{x_0\}\bigr),$
which contradicts Lemma \ref{RC}.

(2) $y_0-x_0\notin C^j$. Then there exist neighborhoods $U_{x_0}$ 
of $x_0$ and $U_{y_0}$ of $y_0$ such that $y'-x'\notin C^j$
for any $x'\in U_{x_0}$ and $y'\in U_{y_0}$. Since 
$x_0\in\omega(x)$ and $y_0\in\omega(y)$, there exist 
$\tau_x,\tau_y>T^j_*$ such that $\Phi_{\tau_x}(x)\in U_{x_0}$ and $
\Phi_{\tau_y}(y)\in U_{y_0},$
and hence
\begin{equation}\label{ry}
    \Phi_{\tau_y}(y)-\Phi_{\tau_x}(x)\notin C^j.
\end{equation}
On the other hand, since $z-x\in\partial C^j$, {\rm(\hyperref[H-2]{H2})} yields $\Phi_{\tau_x}(z)-\Phi_{\tau_x}(x)
\in\operatorname{Int} C^j.$
Since $z\in\omega(y)$, there exists $y_3\in O^+(y)$ close 
to $\Phi_{\tau_x}(z)$, such that
\begin{equation}\label{y4}
    y_3-\Phi_{\tau_x}(x)\in\operatorname{Int} C^j.
\end{equation}
By the connectedness of $O^+(y)$, together with (\ref{ry}) and (\ref{y4}), 
there exists $y_4\in O^+(y)$ such that
$y_4-\Phi_{\tau_x}(x)\in\partial C^j.$
Moreover, (\ref{ry}), (\ref{y4}), and Proposition \ref{OR} imply that $\Phi_{\tau_x}(x)\notin O^+(y).$
Thus, $y_4\in
\partial C^j_{\Phi_{\tau_x}(x)}
\setminus\{\Phi_{\tau_x}(x)\},$
contradicting Lemma \ref{RC}.

Now, we have proved the case (iib) and completed the proof.
\end{proof}

    \begin{prop}\label{coromega}
        Let $x\in \mathcal{R}(\Phi)$. Then for any $y\in \mathcal{R}(\Phi)$, there is $j>0$ such that
		\[
		\omega(y) \setminus \{x\} \subset ( \operatorname{Int}C^j_x) \setminus C^{j-1}_x.
		\]
    \end{prop}
    \begin{proof}
        We consider it in the following two cases: (i) $x\notin \omega(y)$; (ii) $x\in \omega(y)$. 

        (i) $x\notin \omega(y)$. By the claim in Lemma \ref{cor}, there is $j>0$ such that $\omega(y)\subset C^{j}_x\setminus \text{Int}C^{{j}-1}_x$. Together with Lemma \ref{cor}, it has $\omega(y)\subset ( \operatorname{Int}C^j_x)\setminus C^{j-1}_x$. 
        
    (ii) $x\in \omega(y)$. By Proposition \ref{OR}, there is $j>0$ such that $\omega\left(y\right)$ is strongly ordered with $\left \{ C^i \right \} _{i=j}^{N}$, and unordered with $\left \{ C^i \right \} _{i=0}^{j-1}$. Therefore, $\omega(y)\setminus \{x\}\subset ( \operatorname{Int}C^j_x)\setminus C^{j-1}_x$. 
    \end{proof}

	\subsection{Proof of Theorem \ref{limit_set_dichotomy}}\label{ptl}
	\begin{proof}[Proof of Theorem \ref{limit_set_dichotomy}]
	    Let $x,y\in \mathcal{R}(\Phi)$ with $x\neq y$. By Proposition \ref{coromega}, there is $j>0$ such that 
        \begin{equation}\label{cj}
            \omega(y) \setminus \{x\} \subset ( \operatorname{Int}C^j_x)\setminus C^{j-1}_x.
        \end{equation}
        
        We claim that for any $y'\in \omega(y)$ and $x'\in \omega(x)$ with $y'\ne x'$, it has 
        \begin{equation}\label{cj1}
            y'-x'\in C^{j} \setminus \operatorname{Int}  C^{j-1}.
        \end{equation}
        Before proving the claim, we show how it leads to the theorem. By {\rm(\hyperref[H-2]{H2})} and the invariance of $\omega(x)$ and $\omega(y)$, \eqref{cj1} implies $y'-x'\in \operatorname{Int}C^{j}$ and $y'-x'\notin C^{j-1}$ for any $y'\in \omega(y)$ and $x'\in \omega(x)$ with $y'\ne x'$. On the one hand, it has $\omega(x)$ and $\omega(y)$ are strongly ordered with $C^{j}$. Noting $C^{j}\subset C^{i}$ for any $i>j$, it has $\omega(x)$ and $\omega(y)$ are strongly ordered with $\{ C^i \} _{i=j}^{N}$. On the other hand, it has $\omega(x)$ and $\omega(y)$ are unordered with $C^{j-1}$. Noting $C^{i}\subset C^{j-1}$ for any $i<j-1$, it has $\omega(x)$ and $\omega(y)$ are unordered with $\{ C^i \} _{i=0}^{j-1}$.

        It remains to prove the claim. 
        
        First, we show $y'-x'\in C^{j}$ for any $y'\in \omega(y)$ and $x'\in \omega(x)$ with $y'\ne x'$ by contradiction. Suppose there are $x_1 \in \omega(x)$, $y_1 \in \omega(y)$ such that $y_1 - x_1 \notin C^j$. Then, there are $y_2 \in O^+(y)$ with $y_2\neq x$ and $x_2 \in O^+(x)$ such that $y_2 - x_2 \notin C^j$. On the other hand, noting $y_2\in \omega(y)\setminus \{x\}$, it has $y_2-x\in \operatorname{Int} C^{j}$ by \eqref{cj}. Thus by the connectedness of $O^+(x)$, there is $x_3\in O^+(x)$ such that $y_2-x_3\in \partial C^j$. Moreover, by Proposition \ref{OR} it has $y_2\notin\omega(x)$, and hence, $y_2\neq x_3$. Recalling $y_2, x_3\in \mathcal{R}(\Phi)$, it has $y_2\in (\mathcal{R}(\Phi)\cap\partial C^j_{x_3})\setminus\{x_3\}$, which contradicts Lemma \ref{RC}.

        Then, we show $y'-x'\notin \operatorname{Int}  C^{j-1}$ for any $y'\in \omega(y)$ and $x'\in \omega(x)$ with $y'\ne x'$ by contradiction. Suppose there are $x_4 \in \omega(x)$, $y_4 \in \omega(y)$ such that $y_4 - x_4 \in \operatorname{Int}  C^{j-1}$. Then, there are $y_5 \in O^+(y)$ with $y_5\neq x$ and $x_5 \in O^+(x)$ such that $y_5 - x_5 \in \operatorname{Int}  C^{j-1}$. On the other hand, noting $y_5\in \omega(y)\setminus \{x\}$, it has $y_5-x\notin C^{j-1}$ by \eqref{cj}. Thus by the connectedness of $O^+(x)$, there is $x_6\in O^+(x)$ such that $y_5-x_6\in \partial C^{j-1}$. Moreover, by Proposition \ref{OR} it has $y_5\notin\omega(x)$, and hence, $y_5\neq x_6$. Recalling $y_5, x_6\in \mathcal{R}(\Phi)$, it has $y_5\in (\mathcal{R}(\Phi)\cap\partial C^{j-1}_{x_6})\setminus\{x_6\}$, which contradicts Lemma \ref{RC}.
        
        We have proved the claim, which completes the proof.
	\end{proof}	
	
	\section{Structure of the Birkhoff center}\label{S-5}
	
	In this section, we study the structure of the Birkhoff center. Precisely, we present the intersection principle of the Birkhoff center in Subsection \ref{ip}, and prove our main results in Subsection \ref{M1}.

    \subsection{Intersection principle of the Birkhoff center}\label{ip}

    In this subsection, we aim to establish the intersection principle of $\mathcal{B}(\Phi)$ by analyzing the dynamics on the boundary of NICs and present some propositions which play an important role in the proof of the main results. For simplicity, we still label $T_*^i=\operatorname{max}\left \{\tau,t_*^i,T(C^i)\right\}$ for any $1\leq i< N$.
    
	\begin{lem}\label{IPNC}{\rm(Intersection principle of $\mathcal{B}(\Phi)$)}
		 Assume that {\rm(\hyperref[H-1]{H1})}-{\rm(\hyperref[H-2]{H2})} hold. Let $x\in \mathcal{R}(\Phi)$, then $\mathcal{B}(\Phi) \cap (\bigcup_{i=1}^{N} \partial C^i_x) = \{x\}$.
	\end{lem}

    \begin{proof}
	    Suppose on the contrary that, there is $z \in \mathcal{B}(\Phi)$ and $j>0$ such that $z \in\partial C^j_x$ and $z\neq x$. By {\rm(\hyperref[H-2]{H2})}, there exists a $\tau > T_*^j,
        $ such that $\Phi_\tau(z) - \Phi_\tau(x) \in \operatorname{Int} C^j$. Noting $\mathcal{B}(\Phi)$ is invariant, there is $z^*\in \mathcal{B}(\Phi)$ such that $\Phi_{\tau}(z^*)=z$. Also since $x\in \omega(x)$, there is $x^*\in \omega(x)$ such that $\Phi_{\tau}(x^*)=x$. Again, by {\rm(\hyperref[H-2]{H2})}, it has $z^* - x^* \notin C^j$. As $z^* \in \mathcal{B}(\Phi)$, take a sequence $\{z_n\}_{n\geq1} \subset \mathcal{R}(\Phi)$ such that $z_n\rightarrow z^*$ as $n\rightarrow \infty$. Since $z^* - x^* \notin C^j$, without loss of generality, we can assume $z_n - x^* \notin C^j$ for any $n\geq1$. Noting $x^*\in \mathcal{R}(\Phi)$ and $z_n\in \mathcal{R}(\Phi)$, by Theorem \ref{limit_set_dichotomy} for $x^*$ and $z_n$, we have $\omega(z_n) \nsim_{C^j} \omega(x^*) $, and hence, $\omega(z_n)  \nsim_{C^j} \omega(x)$ for any $n\geq1$. In particular, $z_n  \nsim_{C^j} x$ for any $n\geq1$.
        
        On the other hand, since $\Phi_\tau(z) - \Phi_\tau(x) \in \operatorname{Int} C^j$, we also can assume $\Phi_{2\tau}(z_n) - \Phi_{2\tau}(x^*) \in \operatorname{Int} C^j$ for any $n\geq1$. By Theorem \ref{limit_set_dichotomy} again, we obtain $\omega(z_n) \approx_{C^j} \omega(x)$, and hence, $z_n \approx_{C^j}x$ for any $n\geq1$, a contradiction. Thus, we have completed the proof.
	\end{proof}

    \begin{prop}\label{Blayer}
     Let $x\in \mathcal{R}(\Phi)$ and $B$ be a connected component of $\mathcal{B}(\Phi)$. Then there is $j>0$ such that
		\[
		B\setminus \left\{x\right\} \subset (\operatorname{Int} C_x^{j}) \setminus C_x^{j-1}.
		\]
    \end{prop}

    \begin{proof}
    The proposition is trivially true when $B=\{x\}$. Hence, without loss of generality, we assume $B\ne \{x\}$. By virtue of Lemma \ref{IPNC}, to prove this proposition, it suffices to show that there is $j>0$ such that $B \subset C^{j}_x\setminus \operatorname{Int}C^{j-1}_x$. 
    
    Since $X = \bigcup_{i>0} \left(C^i_x\setminus\operatorname{Int}C^{i-1}_x\right)$, there is $j>0$ such that $(B\setminus\{x\})\cap (C^{j}_x\setminus \operatorname{Int}C^{j-1}_x)\ne \emptyset.$ Moreover, by Lemma \ref{IPNC}, it has 
    \begin{equation}\label{B1}
            (B\setminus\{x\} )\cap ((\operatorname{Int}C^{j}_x)\setminus C^{j-1}_x)\ne \emptyset.
    \end{equation}
    Then, we consider the following two cases: (i) $x\notin B$; (ii) $x\in B$. 

    (i) $x\notin B$. It follows from Lemma \ref{IPNC} that $B\cap (\partial C^{j}_x \cup \partial C^{j-1}_x)=\emptyset$, and hence, $B\cap \partial((\operatorname{Int} C_x^{j})\setminus C_x^{j-1})=\emptyset$. Thus by \eqref{B1} and connectedness of $B$, it has $B \subset (\operatorname{Int} C_x^{j}) \setminus C_x^{j-1}.$
        
    (ii) $x\in B$. Suppose, for contradiction, that $B \not\subset C^{j}_x\setminus \operatorname{Int}C^{j-1}_x$. Then there is $k>0$ with $k\ne j$ such that $(B\setminus \{x\})\cap  \left (C^{k}_x\setminus \operatorname{Int}C^{k-1}_x\right )\ne \emptyset$. Without loss of generality, we assume $k>j$. By Lemma \ref{IPNC}, it has $(B\setminus\{x\} )\cap ((\operatorname{Int}C^{k}_x)\setminus C^{k-1}_x)\ne \emptyset$.
    
    Let $B^{j+}=B\cap C_x^j$ and $B^{j-}=B\setminus \operatorname{Int}C_x^j$. Then, $B^{j+}$ and $B^{j-}$ are nonempty closed subsets of $B$. Then $B^{j+}$ and $B^{j-}$ are connected. Indeed, for \(B^{j+}\), suppose, to the contrary, that \(B^{j+}\) is disconnected. Then there exist two disjoint nonempty closed subsets \(B_1^{j+}\) and \(B_2^{j+}\) of \(B^{j+}\) with \(x\in B_1^{j+}\), such that $B^{j+}=B_1^{j+}\cup B_2^{j+}$. Since $B^{j+}\cap B^{j-}=\{x\}$ by Lemma~\ref{IPNC}, it has $(B_1^{j+}\cup B^{j-})\cap B_2^{j+}=\emptyset$. Thus, there are two disjoint closed sets $B_1^{j+}\cup B^{j-}$ and $B_2^{j+}$ such that $B=(B_1^{j+}\cup B^{j-})\cup B_2^{j+}$, which contradicts the connectedness of $B$. A similar argument shows that \(B^{j-}\) is connected.

    Let $m_i\in (B\setminus\{x\}) \cap ((\operatorname{Int}C^{i}_x)\setminus C^{i-1}_x)$ for $i=j,k$. Then $m_k\in B^{j-}$ as $k>j$. Together with the connectedness of $B^{j-}$ and fact $m_j \in \operatorname{Int} C^{j}_x$, there is $p\in B^{j-}$ with $x\ne p$ such that $m_j-p \in \operatorname{Int} C^{j}$ and $p-x \notin \operatorname{Int} C^{j}$. Since $p\in B$, there is $p^*\in \mathcal{R}(\Phi)$ close to $p$ with $p^*\ne x$ such that $m_j-p^* \in \operatorname{Int} C^{j}$ and $p^*-x \notin \operatorname{Int} C^{j}$.
    
    Thus, $p^* \in B^{j+}$. (Otherwise, suppose $p^* \notin B^{j+}$. Let $B^{p}_{1}=B^{j+}\cap  C^j_{p^*}$, $B^{p}_{2}=B^{j+} \setminus \operatorname{Int}C^j_{p^*}$. Then $B^{p}_{1}$ is nonempty as $m_j\in B^{p}_{1}$, $B^{p}_{2}$ is nonempty as $x\in B^{p}_{2}$, and $B^{j+}=B^{p}_{1}\cup B^{p}_{2}$. Moreover, $B^{p}_{1}\cap B^{p}_{2}=\emptyset$, since $B^{j+}\cap \partial C^j_{p^{*}}=\emptyset$ by Lemma \ref{IPNC}. Thus, there are two disjoint closed sets $B^{p}_1$ and $B^{p}_2$ such that $B^{j+}=B^{p}_1\cup B^{p}_2$, which contradicts the connectedness of $B^{j+}$.) Therefore, $p^*-x\in C^{j}$, and hence $p^*-x\in \operatorname{Int}C^j$ by Lemma \ref{IPNC}, which contradicts $p^*-x \notin \operatorname{Int} C^{j}$.
    \end{proof}

\begin{prop}\label{Blayer2}
    Let $B$ be a connected component of $\mathcal{B}(\Phi)$. Then there is $j>0$ such that for any $x\in B$, it has
		$$B \setminus\left\{x\right\}\subset (\operatorname{Int} C_x^{j} )\setminus C_x^{j-1}.$$
    \end{prop}
    
	\begin{proof}
	     If $B$ is a singleton, the result is trivial. Thus, in the remainder, we assume $B$ contains at least two points. 

         First, we claim \textit{there is $j>0$ such that $B \setminus \{x\}\subset (\operatorname{Int} C_{x}^{j}) \setminus C_{x}^{j-1}$ for any $x\in B\cap \mathcal{R}(\Phi)$}. In fact, by Proposition \ref{Blayer}, for any $x\in B\cap \mathcal{R}(\Phi)$, there is $j(x)>0$ such that $B \setminus \{x\}\subset (\operatorname{Int} C_{x}^{j(x)}) \setminus C_{x}^{j(x)-1}$. Thus, we only need to show there is a common index $j$ for any $x\in B\cap \mathcal{R}(\Phi)$. If not, suppose there are $x_1, x_2\in B\cap \mathcal{R}(\Phi)$ and $j_1\neq j_2$ such that $B \setminus \{x_i\}\subset (\operatorname{Int} C_{x_i}^{j_i}) \setminus C_{x_i}^{j_i-1}$ for $i=1,2$. Thus, it has $x_1-x_2\in (\operatorname{Int}C^{j_2})\setminus C^{j_2-1} $ and $ x_2-x_1\in (\operatorname{Int}C^{j_1})\setminus C^{j_1-1}$, a contradiction.

         Then, we show that $B \setminus \{x\}\subset (\operatorname{Int} C_{x}^{j}) \setminus C_{x}^{j-1}$ for any $x\in B\setminus \mathcal{R}(\Phi)$. If not, there are $x\in B\setminus \mathcal{R}(\Phi)$ and $y\in B\setminus \{x\}$ such that (i) $y-x\in C^{j-1}$ or (ii) $y-x\notin \operatorname{Int}C^{j}$. 
         
        (i) $y-x\in C^{j-1}$. By {\rm(\hyperref[H-2]{H2})}, it has $\Phi_{T^{j-1}_*}(y)-\Phi_{T^{j-1}_*}(x)\in \operatorname{Int} C^{j-1}$. Let $x_1\in B\cap \mathcal{R}(\Phi)$ and then $x_1-\Phi_{T^{j-1}_*}(x)\notin C^{j-1}$ by the claim.  Since $\Phi_{T^{j-1}_*}(x)\in B\subset \mathcal{B}(\Phi)$, we can choose $x_2\in \mathcal{R}(\Phi)$ close to $\Phi_{T^{j-1}_*}(x)$ enough, such that $\Phi_{T^{j-1}_*}(y)-x_2\in \operatorname{Int} C^{j-1}$ and $x_1-x_2\notin C^{j-1}$. Then by the connectedness of $B$, $B\cap \partial C_{x_2}^{j-1}\neq\emptyset$. Recall Lemma \ref{IPNC} which implies $\mathcal{B}(\Phi)\cap \partial C_{x_2}^{j-1}=\{x_2\}$. Thus, $x_2\in B\cap \mathcal{R}(\Phi)$. Then by the claim, it has $\Phi_{T^{j-1}_*}(y)\in (\operatorname{Int} C_{x_2}^{j}) \setminus C_{x_2}^{j-1}$, and hence, $\Phi_{T^{j-1}_*}(y)-x_2\notin C^{j-1}$, a contradiction with $\Phi_{T^{j-1}_*}(y)-x_2\in \operatorname{Int} C^{j-1}$.

         (ii) $y-x\notin \operatorname{Int}C^{j}$. By {\rm(\hyperref[H-2]{H2})} and the invariance of $B$, there exist $x'\in O^{-}(x)\cap B$ and $y'\in O^{-}(y)\cap B$ such that $y'-x'\notin C^j.$
If $y'\in \mathcal{R}(\Phi)$, then the claim gives
$y'-x'\in \operatorname{Int} C^j.$
which contradicts $y'-x'\notin C^j$. Thus,
$y'\notin \mathcal{R}(\Phi)$. Let $x_1\in B\cap \mathcal{R}(\Phi)$ and then $y'-x_1\in (\operatorname{Int} C^j)\setminus C^{j-1}$ by the claim.
Since $y'\in B\subset \mathcal{B}(\Phi)$, we can choose $y_2\in \mathcal{R}(\Phi)$ close to $y'$ enough such that
$y_2-x_1\in \operatorname{Int} C^j$ and $y_2-x'\notin C^j.$
Arguing as in (i), we obtain $y_2\in B\cap \mathcal{R}(\Phi)$. Since $x'\in B$, the claim implies $y_2-x'\in \operatorname{Int} C^j,$ contradicting
$y_2-x'\notin C^j$.

         Thus, we have $B \setminus \{x\}\subset (\operatorname{Int} C_{x}^{j}) \setminus C_{x}^{j-1}$ for any $x\in B\setminus \mathcal{R}(\Phi)$. Together with the claim, we have completed the proof.
	\end{proof}

	\subsection{Proof of Main Theorems }\label{M1}

    Now, we are ready to prove our main results.

    \begin{proof}[Proof of Theorem \ref{Thm_2.3}]
    By Proposition \ref{Blayer2}, there is $j>0$ such that for any distinct $x,y\in B$, it has $x-y\in \operatorname{Int}C^j$ and $x-y\notin C^{j-1} $. On the one hand, since $C^j\subset C^i$ for any $i>j$, it follows that $B$ is strongly ordered with $\left \{ C^i \right \} _{i=j}^{N}$. On the other hand, since $C^i\subset C^{j-1}$ for any $i<j-1$, $B$ is unordered with $\left \{ C^i \right \}_{i=0}^{j-1}$.
    \end{proof}

    \begin{proof}[Proof of Theorem \ref{Thm_2.7}]
    Let $B$ be a connected component of $\mathcal{B}(\Phi)$. We first show $x-y\notin\Pi$ for any distinct $x,y\in B$. Fix $s>0$. Since $B$ is invariant, there exist $\tilde{x},\tilde{y}\in B$ such that $\Phi_s(\tilde{x})=x$ and $\Phi_s(\tilde{y})=y$. Moreover, $\Phi_t(\tilde{x}),\Phi_t(\tilde{y})\in B$ for all $t\ge0$. By the virtue of Proposition \ref{Blayer2}, there is $j>0$ such that $a-b\in (\operatorname{Int}C^j)\setminus C^{j-1} $ for any distinct $a,b\in B$. Thus, $\Phi_t(\tilde{x})-\Phi_t(\tilde{y})\in (C^j\setminus C^{j-1})\cup\{0\}$ for any $t\ge 0$. Hence, it has $x-y\notin\Pi$ by Hypothesis {\rm(\hyperref[H-3]{H3})}.
    
    Let $H$ be a $d$-dimensional linear subspace of $X$ with $H\oplus \Pi=X$. Then we can define a projection map $P$ from $B$ to $H$ as
	$$P: B \longrightarrow H;\ x\longmapsto P(x).$$
	Then $P$ is injective. (Otherwise, there are two distinct points $u,v\in B$ such that $P(u)=P(v)$ and thus $u-v\in \Pi$, a contradiction.) Thus, as $B$ is compact and invariant, it follows that $B$ is homeomorphic to a compact invariant set in $\mathbb{R}^{d}$. 	   
    \end{proof}

    \begin{proof}[Proof of Theorem \ref{Thm_2.10}]
    It can be obtained immediately from Theorem \ref{Thm_2.3} and Theorem \ref{Thm_2.7} and the fact that $\text{supp}(\mu)\subset \mathcal{B}(\varphi)$ by \cite[p.28]{M12} or \cite[Appendix A]{J20}.
    \end{proof}

	\begin{proof}[Proof of Theorem \ref{Cor_2.4}]
      Let $\mu$ be an ergodic invariant probability measure. Then by Theorem \ref{Thm_2.10}, there exist a compact connected set $A\subset\mathbb{R}^d$ ($d=1,2$) and a homeomorphism $P:\operatorname{supp}\mu\to A.$ The flow on $\operatorname{supp}\mu$ induces a continuous flow $\Psi_t=P\circ\Phi_t\circ P^{-1}$ on $A$. Note every recurrent point of a continuous flow on an subset of $\mathbb{R}^2$ is either fixed or periodic \cite{SeibertTulley}. Consequently, $\operatorname{supp}\mu$ is either an equilibrium or a periodic orbit of $\Phi$.

       Therefore, $\Phi$ has zero measure-theoretic entropy with respect to $\mu$. Since $\mu$ was arbitrary, the Variational Principle (more details see \cite{D70}) yields the topological entropy of $\Phi$ is zero. This completes the proof.
	\end{proof}

    \begin{rmk}
    Recently, Bochi and Morris \cite{BochiMorris} proved that every recurrent point of a continuous semiflow on a compact subset of $S^2$ is either fixed or periodic. Therefore, the assumption that $\Phi_t$ extends to a flow on $\Gamma$ is no longer needed.
   \end{rmk}

\section{Application}\label{S-6}
	In this section, we present concrete systems that possess NICs structures. Specifically, in Subsection \ref{finite}, we provide an example of bidirectional cyclic feedback systems that exhibit finite NICs. In Subsection \ref{infinite}, we present a parabolic equation in $S^1$ to illustrate that it possesses infinite NICs.

	\subsection{Bidirectional cyclic feedback systems}\label{finite}

    Consider a bidirectional cyclic feedback system on  $\mathbb{R}^n$($n\geq 3$)
    \begin{equation}\label{Three independent variables}
		\dot{x}_i=f_i\left(x_{i-1},x_i,x_{i+1}\right),\quad i=1,2,\dots,n,
	\end{equation}
	 where $x_0=x_n$, $x_{n+1}=x_1$ and the nonlinearity $f$ satisfies the following assumptions: 
     \begin{enumerate}
         \item [(a)] $f=\left(f_1,\dots,f_n\right)$ is of class $C^1$ on $\mathbb R^n$;

     \item [(b)] For each $i=1,\dots,n$, either $$\delta_i \frac{\partial f_i}{\partial x_{i+1}} \ge 0,\ \delta_{i-1} \frac{\partial f_i}{\partial x_{i-1}} > 0;\ \text{ or }\delta_i \frac{\partial f_i}{\partial x_{i+1}} > 0,\ \delta_{i-1} \frac{\partial f_i}{\partial x_{i-1}} \ge 0,$$ where $\delta_i \in \{-1,1\}$ and $\delta_0=\delta_n$;

     \item [(c)] There exists $R>0$ such that
\[
    \langle f(x),x\rangle<0
    \qquad
    \text{for all }x\in\mathbb{R}^n\text{ with }|x|\geq R,
\]
where $\langle\cdot,\cdot\rangle$ denotes the standard inner product
on $\mathbb{R}^n$, and $|\cdot|$ is the induced Euclidean norm.
     \end{enumerate}

      \begin{thm}\label{5-1}
      For a bidirectional cyclic feedback system (\ref{Three independent variables}), any connected component of the Birkhoff center or the support of invariant measures is homeomorphic to a compact invariant set in $\mathbb{R}^2$. Furthermore, its topological entropy is zero.
	 \end{thm}

	Before proving the theorem, we introduce some preliminary material
needed to construct a family of NICs for the bidirectional cyclic feedback system (\ref{Three independent variables}). 
    
    First, we recall the integer-valued Lyapunov function $N(\cdot)$;
see, e.g., \cite{XCYWDZ,MD,FOT,MalletParetSmith1990}. Let $\Delta=\prod_{i=1}^{n} \delta _{i} $. Given $x\in \mathbb{R}^n$ with $x_i\ne 0$ for every $i\in \{1,\dots,n\}$, $N_{\Delta}(x)$ is defined by $$N_\Delta(x) = \#\{\, i\in \{1,\dots,n\} \mid \delta_i x_i x_{i-1} < 0 \,\}+\frac{1+\Delta}{2}.$$ For any $x\in \mathbb{R}^n$, we define  
	 \[
	 N^m_\Delta(x)=\min_{\substack{y \in U_x \\ y_i \neq 0}} N_\Delta(y), \qquad 
	 N^M_\Delta(x)=\max_{\substack{y \in U_x \\ y_i \neq 0}} N_\Delta(y),
	 \]  
	 where \( U_x \) is a sufficiently small neighborhood of \( x \).
    One can extend the domain of $N_{\Delta}(\cdot)$ to an open and dense set $\mathcal{N} = \left\{ x \in \mathbb{R}^n \mid N^m_\Delta(x) = N^M_\Delta(x)\right\}$ by setting $N_\Delta(x) = N^m_\Delta(x) = N^M_\Delta(x)$ for $x\in \mathcal{N}$.

    Then, we consider the linear equation of (\ref{Three independent variables}):
\begin{equation}\label{LE}
\dot{z} = A(t)z, \quad A(t) = \int_{0}^{1} Df(rx^1(t) + (1-r)x^2(t))dr, 
\end{equation}
where \( x^1(t),\, x^2(t) \) are solutions of equation (\ref{Three independent variables}).
Let $z(t)$ be a solution of (\ref{LE}). By \cite[Proposition 2]{FOT} and \cite[Lemma 2.1]{XCYWDZ}, we have the following properties:
    \begin{enumerate}
        \item [($\alpha$)] \( z(t) \in \mathcal{N} \), $t\in \mathbb{R}$ except at finitely many time points;  
    \item [($\beta)$] if \( z(t_0) \in \mathcal{N} \), then \( (z_i(t_0), z_{i-1}(t_0)) \neq (0,0) \) for every $i=1,2,\dots,n$, where $z_0=z_n$ and $z_{n+1}=z_1$;  
    \item [($\gamma$)] if \( z(t_0) \notin \mathcal{N} \cup \{0\} \), then $N^m_\Delta(z(t_0)) = N_\Delta(z(t_0+\epsilon)) < N_\Delta(z(t_0-\epsilon)) = N^M_\Delta(z(t_0))$ for any sufficiently small $\epsilon>0$.
\end{enumerate}

Now, we construct a family of NICs for (\ref{Three independent variables}). Let $\tilde n = n$ if $n$ is odd; \(\tilde n = n+\Delta\) if \(n\) is even. For any $1\leq h\leq \frac{\tilde n-1}{2}$, define the sets  
	 \[
	 K_h = \{0\} \cup \{\, x \in \mathbb{R}^n \mid N^M_\Delta(x) \le 2h-1 \,\}.
	 \]  
	 In particular, define \( K_0 = \{0\} \). Clearly,  $\overline{K_h} = \overline{K_h \setminus \{0\}} $ and $
	 \operatorname{Int} K_h = K_h \setminus \{0\}$ for any $1\leq h\leq \frac{\tilde n-1}{2}$. Moreover, we have
     \begin{equation}\label{kk}
        \{0\}=\overline{K_0}
\subset \cdots \subset
\overline{K_{h-1}}
\subset \overline{K_h}
\subset \cdots \subset
\overline{K_{(\tilde n+1)/2}}=X.
     \end{equation}
 Note that $N_\Delta$ is locally constant on $\mathcal N$ and takes values in
$\{1,3,\ldots,\tilde n\}$. It follows from the definition of $K_h$
and \eqref{kk} that
\begin{equation}\label{2h-1}
    x\in \mathcal{N}\cap(\overline{K_h}\setminus \overline{K_{h-1}})
\quad\Longrightarrow\quad
N_\Delta(x)=2h-1
\end{equation}
for every $1\le h\le (\tilde n+1)/2$.
     By \cite[Theorem 2, Lemma 1]{FOT} and \cite[Propositions 3.1 and 3.2]{XCYWDZ}, we also have the following properties:
    \begin{enumerate}
    \item [($\zeta $)]for  any $1\leq h\leq \frac{\tilde n-1}{2}$, \( \overline{K_h} \) is a solid $ (2h-\frac{1+\Delta}{2})$-cone;
    \item [($\eta$)] for any $t>s\geq0$ and $1\leq h\leq \frac{\tilde n-1}{2}$, we have $S_A(t,s) \bigl( \overline{K_h} \setminus \{0\} \bigr) \subset \operatorname{Int} \overline{K_h}$, where $S_A(t,s)$ is the solution operator of linear equation (\ref{LE}).
    \end{enumerate}

    We now give the proof of Theorem \ref{5-1}.
    \begin{proof}[Proof of Theorem \ref{5-1}.]

       Let $\Phi_t$ be the continuous semiflow generated by the system (\ref{finite}), i.e., $\Phi_t(x_0):=x(t;x_0)$ is a solution of (\ref{finite}) with initial point $x_0$. We show it satisfies hypotheses {\rm(\hyperref[H-1]{H1})}-{\rm(\hyperref[H-3]{H3})}.

        {\rm(\hyperref[H-1]{H1})}: On the one hand, as continuous map is compact in $\mathbb{R}^n$, $\Phi_t$ is compact for any $t>0$. On the other hand, condition (c) implies that the system admits a global attractor. Thus, hypothesis {\rm(\hyperref[H-1]{H1})} is satisfied.
        
        {\rm(\hyperref[H-2]{H2})}: The property ($\eta$) implies for any $x,y\in X$ with $x-y\in \overline{K_h} \setminus \{0\}$ for some $1\leq h\leq \frac{\tilde n-1}{2}$, it has $\Phi_t(x)-\Phi_t(y)\in \operatorname{Int} \overline{K_h}$ as $\Phi_t(x)-\Phi_t(y)=S_A(t,0)(x-y)$ for any $t>0$. Thus, by properties ($\zeta $), ($\eta$) and \eqref{kk}, we obtain $\left\{\overline{K_h}\right\}_{h=0}^{(\tilde n+1)/2}$ is an NICs for $\Phi_t$, and more, $\Phi_t$ is UESM w.r.t $\left\{\overline{K_h}\right\}_{h=0}^{(\tilde n+1)/2}$. Thus, hypothesis {\rm(\hyperref[H-2]{H2})} is satisfied.
        
        {\rm(\hyperref[H-3]{H3})}: 
        Let
\[
\Pi=\{x\in\mathbb{R}^n:x_1=0,\ x_n=0\}.
\]
Clearly, $\Pi$ is a linear subspace of codimension $2$.
Fix $1\le h\le (\tilde n+1)/2$, and let two distinct points $x,y\in\mathcal{B}(\Phi)$ satisfy the premise of {\rm(\hyperref[H-3]{H3})}, that is, 
for any $s>0$, there exist $\tilde{x},\tilde{y}\in\mathcal{B}(\Phi)$ with $\Phi_s(\tilde{x})=x$, $\Phi_s(\tilde{y})=y$ satisfying
\begin{equation}\label{h3}
    \Phi_t(\tilde{x})-\Phi_t(\tilde{y})
\in
\bigl(\overline{K_h}\setminus\overline{K_{h-1}}\bigr)\cup\{0\}
\qquad\text{for all }t\ge0.
\end{equation}
Then, we show $x-y\notin\Pi$.

Suppose, to the contrary, that $x-y\in\Pi$.
Set $z(t)=\Phi_t(\tilde{x})-\Phi_t(\tilde{y})$ for $t\ge0$.
Then $z(s)=x-y\in\Pi.$ Clearly, $z(s)\ne0$ as $x\ne y$. Since $x-y\in\Pi$, it has $z_1(s)=z_n(s)=0$, and hence, $z(s)\notin\mathcal N$ by property $(\beta)$. Thus,
$z(s)\notin\mathcal N\cup\{0\}$. Then, by the continuity of $z$, choose
$\varepsilon>0$ sufficiently small such that
$z(s-\varepsilon)\ne0$ and $z(s+\varepsilon)\ne0$. Moreover,
by property $(\gamma)$, for sufficiently small
$\varepsilon$, it has $$N_\Delta(z(s+\varepsilon))<N_\Delta(z(s-\varepsilon))$$ and $z(s-\varepsilon),\,z(s+\varepsilon)
\in
\mathcal N$. However, by \eqref{h3}, it has $z(s-\varepsilon),\,z(s+\varepsilon)
\in \overline{K_h}\setminus\overline{K_{h-1}}$, and hence, $z(s-\varepsilon),\,z(s+\varepsilon)
\in
\mathcal N\cap
\bigl(\overline{K_h}\setminus\overline{K_{h-1}}\bigr)$. Then, by \eqref{2h-1}, it has $$N_\Delta(z(s-\varepsilon))=N_\Delta(z(s+\varepsilon)),$$ a contradiction.
Therefore $x-y\notin\Pi$, and
hypothesis {\rm(\hyperref[H-3]{H3})} is satisfied.
   
   Thus, system (\ref{finite}) satisfies hypotheses {\rm(\hyperref[H-1]{H1})}-{\rm(\hyperref[H-3]{H3})}. 
   
   Moreover, we observe that the restriction of $\Phi_t$ to the global attractor $\Gamma$ can be extended to a continuous flow. Thus, the theorem is followed by Theorem \ref{Thm_2.3}, \ref{Thm_2.7}, \ref{Thm_2.10}, \ref{Cor_2.4} directly.
    \end{proof}

	\subsection{Scalar parabolic equations on the circle }\label{infinite}

	Consider the following class of scalar parabolic equations on the circle $S^1\triangleq\mathbb{R}/(2\pi\mathbb{Z})$:
	\begin{equation}\label{parabolic}
		\begin{cases}
u_t = u_{xx} + f(x, u, u_x), & x \in S^1, \quad t > 0, \\
u(0, x) = u_0(x), & x \in S^1,
\end{cases}
	\end{equation}
    where the nonlinearity $f$ satisfies the following assumptions: 
    \begin{enumerate}
        \item [(A)] $f:S^1\times\mathbb{R}\times\mathbb{R}\to \mathbb{R}$ is of class $C^2$.
        \item [(B)] There exist a continuous function
\(K:[0,\infty)\to[0,\infty)\) and constants
\(\varepsilon\in(0,2)\) and \(\delta>0\) such that
\begin{align*}
|f(x,\xi,p)|
&\le K(r)\bigl(1+|p|^\varepsilon\bigr),
&& (x,\xi,p)\in S^1\times[-r,r]\times\mathbb{R},
   \quad r>0,\\
\xi f(x,\xi,0)
&<0,
&& (x,\xi)\in S^1\times\mathbb{R},
   \quad |\xi|>\delta.
\end{align*}      
    \end{enumerate}
    
    Let \(A:D(A)=H^2(S^1)\subset L^2(S^1)\to L^2(S^1)\) be defined by
\(Au=-u_{xx}\). Then \(A\) is a nonnegative self-adjoint operator
with compact resolvent. Fix \(a>0\) and set \(G:=A+aI\). Choose
\(\alpha\in(3/4,1)\) and define
\[
X:=D(G^\alpha),\qquad
\|u\|_X:=\|G^\alpha u\|_{L^2(S^1)}.
\]
Since \(D(G^\alpha)=H^{2\alpha}(S^1)\) with equivalent norms and
\(2\alpha>3/2\), the Sobolev embedding theorem yields
\(X\hookrightarrow C^1(S^1)\). Under assumptions
\textnormal{(A)} and \textnormal{(B)}, standard
semilinear parabolic theory, together with the a priori estimates
implied by \textnormal{(B)}, ensures that for every $u_0\in X$,
equation~\eqref{parabolic} admits a unique global solution
\[
u(\cdot;u_0)\in C([0,\infty);X),
\]
which is classical for $t>0$. Consequently, the solution operators
\[
\Phi_t(u_0):=u(t,\cdot;u_0)
\]
define a continuous semiflow
\[
\Phi:[0,\infty)\times X\to X;
\]
see \cite[p.~842]{ppolack}.

    For the semiflow \(\Phi\) generated by equation~\eqref{parabolic},
we have the following result.
     \begin{thm}\label{5-2}
	     Every connected component of the Birkhoff center or the support of invariant measures of a system (\ref{parabolic}) is homeomorphic to a compact invariant set in $\mathbb{R}^2$. Furthermore, the topological entropy of the system is zero.
	 \end{thm}

      The proof reduces to verifying hypotheses {\rm(\hyperref[H-1]{H1})}-{\rm(\hyperref[H-3]{H3})}.
We begin by recalling the fundamental zero-number properties for
solutions of the linear equation governing the difference between
two solutions. These properties yield the natural infinite family
of zero-number NICs $\{\mathcal C^i\}_{i=0}^{\infty}$. Starting from this family, we construct in~\eqref{NCS} a finite
family of perturbed NICs, which will be used to verify hypothesis {\rm(\hyperref[H-2]{H2})}.
	
For \(h\in C^1(S^1)\), let
\[
z(h):=\#\{x\in S^1:h(x)=0\}
\]
denote the number
of distinct zeros of \(h\), with the convention that \(z(h)=+\infty\)
if \(h\) has infinitely many zeros. A point \(x_0\in S^1\) is called
a simple zero of \(h\) if \(h(x_0)=0\) and \(h_x(x_0)\ne0\), and a
multiple zero if \(h(x_0)=h_x(x_0)=0\).
\begin{lem}\label{simpe zeros f}
    Suppose $\{f_n\}_{n\in \mathbb N}\subset C^1(S^1)$ with $z(f_n)\le m$. If there is $f\in C^1(S^1)$ such that $f_n\to f $ in $C^1(S^1)$ as $n\to +\infty$, then $f$ has at most $m$ simple
zeros.
\end{lem}

\begin{proof}
Suppose, to the contrary, that $f$ has more than $m$ simple zeros.
Then there exists a $C^1$-neighborhood $\mathcal U$ of $f$ such that
every $g\in\mathcal U$ has more than $m$ simple zeros.
Since $f_n\to f$ in $C^1(S^1)$, we have $f_n\in\mathcal U$ for all
sufficiently large $n$. Thus, for all
sufficiently large $n$, it has $z(f_n)>m$, contradicting $z(f_n)\le m$.
Therefore, $f$ has at most $m$ simple zeros.
\end{proof}

   Let \(u(t,\cdot)=\Phi_tu_0\) and \(v(t,\cdot)=\Phi_tv_0\),
\(t\ge0\), be two global solutions of equation~\eqref{parabolic}.
Then their difference \(w(t,\cdot):=u(t,\cdot)-v(t,\cdot)\)
satisfies
\begin{equation}\label{linear}
w_t=w_{xx}+b(t,x)w+c(t,x)w_x,
\qquad t>0,\quad x\in S^1,
\end{equation}
where
\begin{align*}
b(t,x)
&=\int_0^1
\partial_2f\bigl(x,ru(t,x)+(1-r)v(t,x),
ru_x(t,x)+(1-r)v_x(t,x)\bigr)\,dr,\\
c(t,x)
&=\int_0^1
\partial_3f\bigl(x,ru(t,x)+(1-r)v(t,x),
ru_x(t,x)+(1-r)v_x(t,x)\bigr)\,dr.
\end{align*}

Let \(T^t_{(u_0,v_0)}:X\to X\) denote the solution operator
associated with equation~\eqref{linear}. According to
\cite[Theorems~3.4.4 and 7.1.3 and Exercise~10 on p.~196]{He},
\(T^t_{(u_0,v_0)}\) is compact for every \(t>0\), and the map
\[
(u_0,v_0,t)\longmapsto T^t_{(u_0,v_0)}
\]
is continuous from \(X\times X\times(0,\infty)\) into
\(\mathcal L(X)\), where $\mathcal L(X)$ denotes the Banach space of bounded linear
operators on $X$, equipped with the operator norm. Moreover,
\begin{equation}\label{u-v}
\Phi_tu_0-\Phi_tv_0
=
T^t_{(u_0,v_0)}(u_0-v_0),
\qquad u_0,v_0\in X,\quad t>0.
\end{equation}
By \cite[Theorems~B and C]{An}, the following properties hold:
\begin{enumerate}
\item[(a)]
If \(w(0,\cdot)\not\equiv0\), then \(z(w(t,\cdot))<\infty\)
for every \(t>0\).
\item[(b)]
For \(0<s<t\),
\[
z(w(t,\cdot))\le z(w(s,\cdot)).
\]
\item[(c)]
If \(w(t_0,\cdot)\) has a multiple zero for some \(t_0>0\), then
\[
z(w(t_0-\varepsilon,\cdot))
\ge z(w(t_0+\varepsilon,\cdot))+2
\]
for every sufficiently small \(\varepsilon\in(0,t_0)\).

\item[(d)]
If \(w(0,\cdot)\not\equiv0\), then, for every \(s>0\), there are
only finitely many \(t>s\) for which \(w(t,\cdot)\) has a multiple
zero.
\end{enumerate}

We now define the natural zero-number NICs
\(\{\mathcal C^i\}_{i=0}^{N}\) in \(X\), where \(N=+\infty\).
Set \(\mathcal C^0:=\{0\}\), \(\mathcal C^N:=X\), and, for
\(1\le i<N\), define
\[
	\mathcal C^i
:=
\overline{
\left\{
h\in X:
(h(x),h_x(x))\ne(0,0)\ \text{for every }x\in S^1,
\quad z(h)\le 2(i-1)
\right\}
}.
\]
By \cite[Section~7]{Tl}, the following properties hold:
\begin{enumerate}
\item[(\(\alpha\))]
For every \(1\le i<N\), \(\mathcal C^i\) is a solid
\((2i-1)\)-cone.

\item[(\(\beta\))]
For every \(v\in X\setminus\{0\}\), \(t>0\), and
\((u_0,v_0)\in X\times X\), there exists \(j\) with
\(1\le j<N\) such that
\[
T^t_{(u_0,v_0)}v\in\mathcal C^j.
\]

\item[(\(\gamma\))]
For every \(1\le i<N\), \(t>0\),
\((u_0,v_0)\in X\times X\), and
\(v\in\mathcal C^i\setminus\{0\}\), there exists an open
neighborhood \(U\) of \(v\) such that
\[
T^t_{(u_0,v_0)}(U)
\subset\mathcal C^i\setminus\{0\}.
\]

\item[(\(\eta\))]
There exists a codimension-two linear subspace \(\Pi\subset X\)
such that, for every \(1\le i<N\), \(t_1>0\), \(w\in X\), and
\((u_0,v_0)\in X\times X\),
\[
T^{t_1}_{(u_0,v_0)}w
\in
(\mathcal C^i\setminus\mathcal C^{i-1})\cap\Pi
\quad\Longrightarrow\quad
w\notin\mathcal C^i.
\]
\end{enumerate}

\begin{rmk}\label{counterexample}
Although \(\{\mathcal C^i\}_{i=0}^N\) forms a family of NICs
for \(\Phi_t\), the semiflow \(\Phi_t\) is not UESM with respect
to this family; see Appendix \ref{app:counterexample} for a counterexample. Therefore,
the natural zero-number NICs cannot be used directly to verify
hypothesis~{\rm(\hyperref[H-2]{H2})}. The counterexample also shows that
there exists \(j\) with \(1\le j<j+1<N\) such that
\[
\mathcal C^j
\not\subset
\operatorname{Int}\mathcal C^{j+1}\cup\{0\}.
\]
However, by \cite[Proposition~6.1]{Tl}, there exist a positive
integer \(N_0\) and a family of NICs
\(\{C^i\}_{i=0}^{N_0+1}\), constructed from
\(\{\mathcal C^i\}_{i=0}^N\), with respect to which
\(\Phi_t\) satisfies hypothesis~{\rm(\hyperref[H-2]{H2})}; see
Lemma~\ref{XPI} below.
\end{rmk}

Before giving the formal proof of Theorem \ref{5-2}, we still need the following important lemma, which will be very useful for our subsequent proof.

\begin{lem}\label{XPI}
For system (\ref{parabolic}), let $\Pi$ be the subspace specified in property $(\eta)$. Let $\Pi'$ be a complementary subspace of $\Pi$, so that $X=\Pi\oplus\Pi'$, and let $P:X\to \Pi'$ denote the associated projection. For $s\ge 0$, define
\[
P_s=\left\{u\in X:\|P u\|_X \le s\,\|(I-P)u\|_X\right\}.
\]
Then there exist $T_0>0$, $\delta>0$, $\lambda\in[0,1)$, a positive integer $N_0$, and a sequence of sets
\begin{equation}\label{NCS}
    C^0=\{0\},\quad C^{N_0+1}=X,\quad C^i \subset \operatorname{Int} C^{i+1} \cup \{0\} \subset C^{i+1},\quad 1\le i\le N_0
\end{equation}
such that each $C^i$ is a $(2i-1)$-cone and the following properties hold:

\begin{enumerate}
\item[(i)] If $x-y\in C^i\setminus\{0\}$ for an $1\le i\le N_0$, then for all $t\ge T_0$,
\[
\Phi_t x-\Phi_t y \in \operatorname{Int} C^i.
\]

\item[(ii)] If $x-y\in C^i\setminus C^{i-1}$ and $\Phi_{2T_0}x-\Phi_{2T_0}y\in C^i\setminus C^{i-1}$ for an $1\le i\le N_0$, then
\[
\Phi_{T_0}x-\Phi_{T_0}y \notin P_\delta.
\]
\item[(iii)] If $\Phi_t(x)-\Phi_t(y)\notin C^{N_0}$
for some $t\ge T_0$, then
\[
\|\Phi_t(x)-\Phi_t(y)\|_X
\le \lambda^t\|x-y\|_X.\]
\end{enumerate}
\end{lem}

\begin{proof}
    In view of~(\ref{u-v}) and properties~$(\alpha)$, $(\beta)$,
$(\gamma)$, and~$(\eta)$, system~(\ref{parabolic}) satisfies all the
hypotheses of Proposition~6.1 in~\cite{Tl}.
Applying Proposition~6.1 in~\cite{Tl} yields $T_0>0$, $\delta>0$, $\lambda\in[0,1)$, a positive
integer $N_0$, and a family of cones satisfying~(\ref{NCS}) and
properties 
\textnormal{(i)}-\textnormal{(iii)}.
This completes the proof.
\end{proof}

With the above preparations, we now prove Theorem \ref{5-2}.    
\begin{proof}[Proof of Theorem \ref{5-2}.]
    Recall that $\Phi_t$ is the continuous semiflow generated by system (\ref{parabolic}).
    Let
$\{ C^i\}_{i=0}^{N_0+1}$ be the NICs for $\Phi_t$ given by
(\ref{NCS}).
We verify hypotheses {\rm(\hyperref[H-1]{H1})}, {\rm(\hyperref[H-2]{H2})}, and {\rm(\hyperref[H-3]{H3})}.
        
     First, we verify hypothesis~{\rm(\hyperref[H-1]{H1})}.
Equation~\eqref{parabolic} can be written as
\begin{equation}\label{parabolic1}
u_t+Gu=(F+aI)(u),
\end{equation}
where \(G=A+aI\) and \(F(u)(x)=f(x,u(x),u_x(x))\).
Since \(f\in C^2\) and \(X\hookrightarrow C^1(S^1)\), the map
\(F+aI:X\to L^2(S^1)\) is locally Lipschitz continuous.
Moreover, \(G\) is a positive sectorial operator with compact
resolvent. Therefore, by \cite[Theorem~3.3.1]{CT},
\(\Phi_t:X\to X\) is compact for every \(t>0\).
Combining this compactness with the dissipativity condition
\textnormal{(B)}, we conclude that \(\Phi\) possesses a global
attractor \(\Gamma\subset X\). Hence hypothesis~{\rm(\hyperref[H-1]{H1})}
is satisfied.

        Second, we show that hypothesis~{\rm(\hyperref[H-2]{H2})} holds.
By Lemma~\ref{XPI}(i), $\Phi_t$ is UESM w.r.t.\ $C^i$ for every $1\leq i\leq N_0$, with the common UESM time $T_0$.
Hence, $\Phi_t$ is UESM with respect to the NICs $\{C^i\}_{i=0}^{N_0+1}$, and hypothesis~{\rm(\hyperref[H-2]{H2})} holds.

     Finally, we verify hypothesis {\rm(\hyperref[H-3]{H3})} with the
codimension two subspace $\Pi$ specified in property $(\eta)$. 

For $1\le i\le N_0$, and let two distinct points $x,y\in\mathcal{B}(\Phi)$
satisfy the premise of {\rm(\hyperref[H-3]{H3})}, that is, 
for any $s>0$, there exist $\tilde{x},\tilde{y}\in\mathcal{B}(\Phi)$ with $\Phi_s(\tilde{x})=x$, $\Phi_s(\tilde{y})=y$ satisfing
\begin{equation*}
    \Phi_t(\tilde{x})-\Phi_t(\tilde{y})
\in
(C^i\setminus C^{i-1})\cup\{0\}
\quad\text{for all }t\ge0.
\end{equation*}
Taking $s=T_0$.
By Lemma~\ref{XPI}(i), it has $\Phi_{2T_0}(\tilde{x})-\Phi_{2T_0}(\tilde{y})
\in \operatorname{Int}C^i$, and hence, $\Phi_{2T_0}(\tilde{x})-\Phi_{2T_0}(\tilde{y})
\in C^i\setminus C^{i-1}$. Thus, by Lemma~\ref{XPI}(ii), it has $\Phi_{T_0}(\tilde{x})-\Phi_{T_0}(\tilde{y})
\notin P_\delta$. Since $\Pi\subset P_\delta$, it follows that $x-y\notin\Pi.$
Thus, hypothesis {\rm(\hyperref[H-3]{H3})} holds for
$1\le i\le N_0$.

For $i= N_0+1$, we claim that there do not exist two distinct points of $\mathcal{B}(\Phi)$ satisfying the premise of {\rm(\hyperref[H-3]{H3})}. Consequently, hypothesis {\rm(\hyperref[H-3]{H3})} is naturally satisfied for $i= N_0+1$. Now, we prove the claim. Suppose, to the contrary, there are two distinct points $x,y\in\mathcal{B}(\Phi)$ satisfying the premise of {\rm(\hyperref[H-3]{H3})} for $i=N_0+1$. For any $s\ge T_0$, let
$x_s,y_s\in\mathcal{B}(\Phi)$ such that $\Phi_s(x_s)=x$ and $\Phi_s(y_s)=y$. Since $x-y
\notin C^{N_0}$,
Lemma~\ref{XPI}(iii) gives $\|x-y\|_X
\le
\lambda^s\|x_s-y_s\|_X.$
Since $\mathcal{B}(\Phi)$ is compact, there exists $M>0$ such that $\|u-v\|_X\le M$ for all $u,v\in\mathcal{B}(\Phi)$. Hence,
\[
\|x-y\|_X\le \lambda^s M,
\]
which contradicts the fact $x\neq y$, since $s$ is arbitrary. Thus, we have proved the claim, and hence, hypothesis {\rm(\hyperref[H-3]{H3})} also holds for $i= N_0+1$.

Therefore, system~\eqref{parabolic} satisfies hypotheses~{\rm(\hyperref[H-1]{H1})},
{\rm(\hyperref[H-2]{H2})}, and~{\rm(\hyperref[H-3]{H3})}. Thus, by Theorem \ref{Thm_2.7}, every connected component of the Birkhoff center of a system (\ref{parabolic}) is homeomorphic to a compact invariant set in $\mathbb{R}^2$.

Next, we observe that the restriction of $\Phi$ to the global attractor
$\Gamma$ extends to a continuous flow. Indeed, property~\textnormal{(a)} implies that $\Phi_t$ is
injective for every $t>0$. Otherwise, if $\Phi_t u=\Phi_t v$ for some
$u\ne v$, then the difference $\Phi_t u-\Phi_t v$ would vanish
identically and hence have infinitely many zeros, contradicting
property~\textnormal{(a)}. On the other hand, the strict invariance
$\Phi_t(\Gamma)=\Gamma$ implies that $\Phi_t|_\Gamma$ is surjective.
Since $\Gamma$ is compact and $\Phi_t|_\Gamma$ is a continuous
bijection, it is a homeomorphism for every $t>0$. Define
\[
\Phi_{-t}|_\Gamma:=\bigl(\Phi_t|_\Gamma\bigr)^{-1},
\qquad t>0.
\]
The compactness of $\Gamma$, together with the continuity and semigroup
property of $\Phi$, ensures that the resulting extension is a
continuous flow on $\Gamma$. Therefore, by the virtue of Theorem~\ref{Cor_2.4}, the topological entropy of system~\eqref{parabolic} is
zero. 

This completes the proof.
\end{proof}

\appendix
\renewcommand{\thesection}{\Alph{section}}
\setcounter{section}{1}   

\setcounter{section}{0}
\renewcommand{\thesection}{\Alph{section}}
\refstepcounter{section}

\section*{Appendix \thesection: Hausdorff distance and separation index}
\label{app:hausdorff}

\addcontentsline{toc}{section}
{Appendix \thesection: Hausdorff distance and separation index}

\setcounter{equation}{0}
\renewcommand{\theequation}
{\thesection.\arabic{equation}}

\setcounter{thm}{0}
\renewcommand{\thethm}
{\thesection.\arabic{thm}}

In this appendix, we present some notions and properties about Hausdorff distance and separation index, which are used in Section \ref{IRSaNB}.

For any two closed subsets $A,B\subset X$, the \textit{Hausdorff distance} between $A$ and $B$ is defined as $$d_{H}(A,B)=\max\{\sup_{a\in A} \inf_{b\in B}\left \|  a-b \right \|_X,\sup_{b\in B} \inf_{a\in A}\left \|  a-b \right \|_X\}.$$
The \textit{separation index} between $A$ and $B$ is defined as $${\rm \underline{dist}}(A,B)=\inf_{x\in A,y\in B}\left \|  x-y \right \|_X.$$
When $A$ is a compact set and $B$ is a closed set, it has ${\rm\underline{dist}}(A,B)>0$ if and only if $A\cap B=\emptyset$. Let $A,B,C\subset X$ be closed sets, then
\begin{align}
	&d_{H}(x+A,y+B)\leq \left \| x-y \right \|_X+d_{H}(A,B) ;\label{dis_1}\\
	&{\rm\underline{dist}}(A,B)\leq{\rm\underline{dist}}(A,C)+d_H(B,C).\label{dis_2}
\end{align}
Moreover, we have 

\begin{lem}\label{AA-dist}
	For a compact set A in $X$, $d_H\left(\Phi_t A,A\right)$ is continuous w.r.t $t$.
\end{lem}
\begin{proof}
See the proof in \cite[Lemma 3.1]{SWZ}.
\end{proof}

\begin{lem}\label{A-dist}
	Fix $t_1,t_2>0$. Let $A\subset X$ be compact and $B\subset X$ be closed. Then $${\rm \underline{dist}}\left ( \Phi _{t_1}\left ( x \right ) -\Phi _{t_2} \left ( x+A   \right ) ,B  \right )$$ is continuous w.r.t $x$. 
\end{lem}

\begin{proof}
	 For any $x,y\in X$, it has
	\begin{equation*}
		\begin{split}
 		& \left | {\rm \underline{dist}}\left ( \Phi _{t_1}\left ( y \right ) -\Phi _{t_2} \left ( y+A   \right ) ,B  \right ) -{\rm \underline{dist}}\left ( \Phi _{t_1}\left ( x \right ) -\Phi _{t_2} \left ( x+A   \right ) ,B  \right )  \right |  \\ \overset{\eqref{dis_2}}{\leq} &d_{H}\left(\Phi _{t_1}\left ( y \right ) -\Phi _{t_2} \left ( y+A   \right ), \Phi _{t_1}\left ( x \right ) -\Phi _{t_2} \left ( x+A   \right )\right) \\\overset{\eqref{dis_1}}{\leq}&\left \| \Phi _{t_1}\left ( y \right )- \Phi _{t_1}\left ( x \right )\right \|_X+d_{H}\left ( \Phi _{t_2} \left ( y+A  \right ) ,\Phi _{t_2} \left ( x+A   \right )\right ) .
		\end{split}
	\end{equation*}
    Since $\Phi_{t_1}$ is continuous, it has $\|\Phi_{t_1}(y)-\Phi_{t_1}(x)\|_X\rightarrow 0$ as $y\rightarrow x$. Together with a fact that $d_{H}\bigl( \Phi_{t_2}(y+A), \Phi_{t_2}(x+A) \bigr) 
	\le \sup_{a \in A} \bigl\| \Phi_{t_2}(y+a) - \Phi_{t_2}(x+a) \bigr\|_X$, 
    the proof of the lemma is complete once we establish that
    \begin{equation}\label{sup0}
        \sup_{a\in A}
\left\|
\Phi_{t_2}(y+a)-\Phi_{t_2}(x+a)
\right\|_X
\rightarrow 0
\ \text{as }y\to x.
    \end{equation}
    
   It remains to prove (\ref{sup0}). If not, there exist $\varepsilon_0>0$ and $\{y_n\}_{n\in \mathbb{N}}$ with $y_n\rightarrow x$ as $n\rightarrow\infty$ such that $\sup_{a \in A} \bigl\| \Phi_{t_2}(y_n+a) - \Phi_{t_2}(x+a) \bigr\|_X>\varepsilon_0$, and hence, there exists $\{a_n\}_{n\in \mathbb{N}}\subset A$ such that \begin{equation}\label{a^*}
        \bigl\| \Phi_{t_2}(y_n+a_n) - \Phi_{t_2}(x+a_n) \bigr\|_X>\frac{\varepsilon_0}{2}.
    \end{equation} 
    On the other hand, since $A$ is compact, there is a point $a$ and subsequence $\{a_i\}_{i\in \mathbb{N}}$ of $\{a_n\}_{n\in \mathbb{N}}$ such that $a_i\rightarrow a$ as $i\rightarrow\infty$. Thus by the continuity of $\Phi_{t_2}$, it has $\bigl\| \Phi_{t_2}(y_i+a_i) - \Phi_{t_2}(x+a_i) \bigr\|_X \rightarrow 0$ as $i\rightarrow \infty$, contradicting (\ref{a^*}). Hence, we have proved (\ref{sup0}) and the proof is complete.
\end{proof}

\refstepcounter{section}

\section*{Appendix \thesection: A counterexample}
\label{app:counterexample}

\addcontentsline{toc}{section}
{Appendix \thesection: A counterexample}

\setcounter{equation}{0}

\newtheorem{lemma}[thm]{Lemma}
\newtheorem{corollary}[thm]{Corollary}

\setcounter{thm}{0}

In this appendix, we give the counterexample mentioned in Remark \ref{counterexample} of subsection \ref{infinite}. We show that, for the heat equation
\[
\begin{cases}
u_t=u_{xx}, & t>0,\ x\in S^1,\\
u(0,x)=u_0(x), & x\in S^1,
\end{cases}
\]
the generated continuous semiflow $\Phi_t$ is not UESM w.r.t the natural zero-number NICs \(\{\mathcal C^i\}_{i=0}^{N}\).

For any $b>0$, let
\[
u_0^b(x)
=
e^b\cos x-\frac14 e^{4b}\cos 2x-\frac34.
\]
We claim that: $$ \textit{ For any }b>0, \textit{ there is }u_0^b\in \mathcal{C}^3, \textit{ but } \Phi_bu_0^b\in\partial \mathcal{C}^3.$$
This claim immediately implies $\Phi_t$ is not UESM w.r.t \(\{\mathcal C^i\}_{i=0}^{N}\).

\begin{proof}[Proof of the claim.]
Clearly, $u_0^b$ has exactly four zeros on $S^1$, all of which are simple, and hence $u_0^b\in \mathcal{C}^3$ for any $b>0$. Thus, it remains to prove $\Phi_bu_0^b\in\partial \mathcal{C}^3$ for any $b>0$. Fix $b>0$. By property ($\gamma$) in Section \ref{infinite}, it has $\Phi_tu_0^b\in \mathcal{C}^3$ for $t\geq0$, and hence, $\Phi_bu_0^b\in \mathcal{C}^3$. 

Then, we show there is a sequence $\{G_n\}_{n\in \mathbb{N}}$ in $X$ such that $G_n\notin\mathcal{C}^3$ for all sufficiently large $n$, and $$\left \|G_n- \Phi_bu_0^b \right \| _X\rightarrow 0\ \text{ as } n\rightarrow +\infty.$$

Let $$G_n(x)= \bigl(\Phi_b u^b_0\bigr)(x) + g_n(x),$$ where \(g_n(x) = \frac{\sin nx}{n^3}\) in \(X\), then $\left \|G_n- \Phi_bu_0^b \right \| _X\rightarrow 0$ as $n\rightarrow +\infty$, since $\left \|g_n\right \| _X\rightarrow 0$ as $n\rightarrow +\infty$. In fact, as Fourier coefficients of the function $g_n(x)$ are given by
	\[
	\hat{g}_n(k) = \frac{1}{2\pi} \int_0^{2\pi} \left( \frac{e^{inx} - e^{-inx}}{2i n^3} \right) e^{-ikx} \,dx,
	\]
	with a fact that the Fourier coefficients of \((A+aI)^\alpha g_n\) are \((k^2 + a)^\alpha \hat{g}_n(k)\) (as \((A+aI)^\alpha e^{ikx} = (k^2 + a)^\alpha e^{ikx}\)), by Parseval's identity, it has
	\[
	\|g_n\|_X^2 = 2\pi\sum_{k=-\infty}^{+\infty} \left| (k^2 + a)^\alpha \hat{g}_n(k) \right|^2
	= \pi\frac{(n^2 + a)^{2\alpha}}{n^6} \to 0 \ \text{as } n \to +\infty.
	\]

Now we prove $G_n\notin\mathcal{C}^3$ for all sufficiently large $n$. For any $n>0$, let 
	\[
	M_n=\left\{ x_{m,n} = \frac{\pi/2 + m\pi}{n} : m \in \mathbb{N},\; x_{m+1,n}< \frac{1}{n^\alpha}  \right\}.
	\]
	A direct calculation gives, for any point $x_{m,n}\in M_n$
	\[
	G_n(x_{m,n}) \in \left( -\frac{1}{8n^{4\alpha}} + (-1)^m \frac{1}{n^3},\; (-1)^m \frac{1}{n^3} \right],
	\]
	which implies \(G_n(x_{m,n}) G_n(x_{m+1,n}) < 0\). Thus, for any $n>0$ and $x_{m,n}\in M_n$, there exists \(\theta_{m,n} \in (x_{m,n}, x_{m+1,n})\) such that \(G_n(\theta_{m,n}) = 0\). 
    
We assert that there exists $N_1>0$ such that, for any $n>N_1$ and $x_{m,n}\in M_n$, $\theta_{m,n}$ is a simple zero of $G_n$. If not, suppose there exists a sequence $n_k\to +\infty$ as $k\to +\infty$ and $x_{m_k,n_k}\in M_{n_k}$ such that $G_{n_k}'(\theta_{m_k,n_k})=0$. Then \begin{equation}\label{xm1}
    \bigl(\Phi_b u^b_0\bigr)(\theta_{m_k,n_k})+\frac{\sin(n_k\theta_{m_k,n_k})}{n_k^3}=0,\quad\bigl(\Phi_b u^b_0\bigr)'(\theta_{m_k,n_k})+\frac{\cos(n_k\theta_{m_k,n_k})}{n_k^2}=0.
\end{equation}
Note that
    $$\left |\bigl(\Phi_b u^b_0\bigr)(\theta_{m_k,n_k})\right |=\frac12(1-\cos\theta_{m_k,n_k})^2\le \frac{1}{8}\left(\theta_{m_k,n_k}\right )^4$$
and 
\[
\left|\bigl(\Phi_bu_0^b\bigr)'(\theta_{m_k,n_k})\right|
=(1-\cos\theta_{m_k,n_k})\sin\theta_{m_k,n_k}
\le
\frac{1}{2}(\theta_{m_k,n_k})^3.
\]
Together with (\ref{xm1}) and $\theta_{m_k,n_k}<x_{m_k+1,n_k}<\frac{1}{n_k^\alpha}$, it has
$$|\sin(n_k\theta_{m_k,n_k})|
< \frac18 n_k^{3-4\alpha},\quad |\cos(n_k\theta_{m_k,n_k})|
< \frac12 n_k^{2-3\alpha}.$$
Since $\alpha > \frac{3}{4}$, there exists $N>0$ such that for all $n_k > N$, it has
$$|\sin(n_k\theta_{m_k,n_k})| < \frac{1}{2}, \quad |\cos(n_k\theta_{m_k,n_k})| < \frac{1}{2}.$$
Consequently, it has
$$\sin^2(n_k\theta_{m_k,n_k})+\cos^2(n_k\theta_{m_k,n_k})
< \frac12,$$
contradicting the identity $\sin^2(\cdot)+\cos^2(\cdot) =1$.
Therefore, the assertion holds.

Consequently, for any $n>N_1$, $G_n$ has simple zero in $(x_{m,n}, x_{m+1,n})$, and thus 
$$\#\{\theta: G_{n}(\theta)=0, G'_{n}(\theta)\neq0\}\geq\#M_n-1\ge\frac{n^{1-\alpha}}{\pi}-3,$$
where $\#$ is denoted the amount of elements of a set. Thus, for any 
$$n > \max \left\{ N_1, \left( 8\pi \right)^{\frac{1}{1-\alpha}} \right\},$$ $G_n$ has more than four simple zeros on $S^1$, and hence $G_n\notin\mathcal{C}^3$ by Lemma \ref{simpe zeros f}.

Thus, we have obtained $\Phi_b u^b_0 \in \partial \mathcal{C}^3$ and completed the proof.
\end{proof}

\end{document}